\documentclass{elsarticle}
\biboptions{sort&compress}
\usepackage{microtype}
\usepackage{amsmath,amsfonts,amssymb}
\usepackage{amsthm}
\usepackage{xcolor}
\usepackage{graphicx}

\usepackage{tikz}
\usepackage{pgfplots}
\pgfplotsset{compat=newest}
\usetikzlibrary{plotmarks}
\usetikzlibrary{arrows.meta}
\usepgfplotslibrary{patchplots}
\usepackage{grffile}

\usepackage{hyperref}

\newtheorem{theorem}{Theorem}[section]

\newtheorem{lemma}[theorem]{Lemma}
\newtheorem{ass}[theorem]{Assumption}
\newtheorem{remark}{Remark}

\newcommand{\bb}{\boldsymbol}
\newcommand{\rme}{\mathrm{e}}
\newcommand{\RR}{\mathbb{R}}

\begin{document}
\myfooter[L]{}

\begin{frontmatter}
\title{The turnpike control in stochastic multi-agent dynamics: a discrete-time approach with exponential integrators}
\author[1,2]{Fabio Cassini}
\ead{fabio.cassini@univr.it, cassini@altamatematica.it}
\author[3]{Chiara Segala\corref{cor1}}
\ead{chiara.segala@usi.ch}
\cortext[cor1]{Corresponding author}
\affiliation[1]{organization={Department of Computer Science,
    University of Verona},addressline={Strada Le Grazie, 15},
  postcode={37134},
  city={Verona},
country={Italy}}
\affiliation[2]{organization={Istituto Nazionale di Alta Matematica},
  addressline={Piazzale Aldo Moro, 5},
  postcode={00185},
  city={Roma},
  country={Italy}}
\affiliation[3]{organization={Faculty of Informatics,
    Universita della Svizzera Italiana -- USI },
    addressline={Via La Santa, 1},
    postcode={6962},
    city={Lugano},
country={Switzerland}
}

\begin{abstract}
In this manuscript, we study the turnpike property in stochastic discrete-time optimal control problems for interacting agents. Extending previous deterministic results, we show that the turnpike effect persists in the presence of noise under suitable dissipativity and controllability conditions. To handle the possible stiffness in the system dynamics, we employ for the time discretization, integrators of exponential type. Numerical experiments validate our findings, demonstrating the advantages of exponential integrators over standard explicit schemes and confirming the effectiveness of the turnpike control even in the stochastic setting. 
\end{abstract}


%\paragraph{AMS Classification:} {{93C55\ \textbullet\ 93A16\ \textbullet\ 49M25\ \textbullet\ 65L04 \textbullet\ 65C30 \textbullet\ 37M15}}

%\maketitle
\end{frontmatter}
\section{Introduction}\label{sec:intro}
The collective behavior of interacting agents has been the subject of extensive mathematical research, see, e.g., \cite{MR3642940, MR3969953}. These models are typically governed by systems of differential equations that describe various interaction mechanisms, such as attraction and repulsion, as well as control inputs that regulate system behavior. This makes them applicable to a wide range of fields, including biology, engineering, economics, and sociology \cite{application-[14],application-[29]}. Significant attention has been devoted to modeling and analysis, and also the study of control impact for such systems has gained interest more recently \cite{sparse-Cucker}. Controlling large-scale, high-dimensional agent-based systems presents several theoretical and computational challenges, which have been addressed using various techniques such as Riccati-based methods \cite{MR3431287,MR4469721}, moment-driven control \cite{MR4399019}, and model predictive control (MPC) \cite{MR3894072}, among others.

A promising approach to efficiently control these complex systems is the so-called turnpike control, which leverages the turnpike property to reduce computational costs while maintaining control effectiveness. In few words, the turnpike phenomenon states that, for sufficiently long time horizons, the optimal control and state remain close to the solution of a corresponding static optimization problem over the majority of the time interval. This property has been widely used to improve computational efficiency in optimal control problems \cite{MR3973342} and has played a key role in establishing convergence results for MPC methods \cite{MR4402854}.  

In deterministic settings, the turnpike property has been established for both continuous-time \cite{main} and discrete-time \cite{turnpike-2, gruene2018turnpike,GugatHertyLiuSegala2024} systems. However, real-world applications frequently involve stochastic perturbations, which can significantly impact the system's long-term behavior and control strategies. In this work, we extend the turnpike analysis to stochastic dynamics \cite{ou2021simulation,10759735}, addressing the challenges posed by noise in multi-agent systems \cite{ma2017consensus,albi2024robust}. Additionally, unlike in~\cite{GugatHertyLiuSegala2024} where the authors employed the explicit Euler method for the time marching, we use integrators of exponential type~\cite{HO10}. This choice, in particular, allows to effectively handle the inherent stiffness that may arise in the interaction dynamics. Exponential integrators provide improved stability properties compared to standard explicit schemes, making them particularly well-suited for stiff systems, and to the authors' knowledge they have not been employed in this context yet.  

The main contributions of this paper are as follows. First, we establish the turnpike property for stochastic discrete-time optimal control problems governing interacting agents, extending previous deterministic results \cite{main, GugatHertyLiuSegala2024}. Our analysis demonstrates that the turnpike effect persists even in the presence of noise, provided certain dissipativity and controllability conditions are fulfilled.
Second, we introduce the use of exponential Rosenbrock integrators for the numerical discretization of the obtained stochastic control problem. To this end, we in particular reformulate the problem in matrix form, which facilitates the application of these integrators. 
Finally, we conduct a series of numerical experiments to validate our findings. We compare standard explicit and exponential integrator-based discretizations in both controlled and uncontrolled settings, demonstrating that schemes of exponential type exhibit superior stability properties, especially in the presence of stiffness. Moreover, we confirm that the turnpike control strategy remains effective in steering the agents toward the desired asymptotic behavior, even in the stochastic setting.

The remainder of this paper is structured as follows. In Section~\ref{sec:stochcontrol}, we introduce the stochastic control problem under consideration, including the system dynamics and cost functional. In Section~\ref{sec:timedisc}, we present the time discretization using exponential integrators, highlighting their advantages over explicit schemes. Section~\ref{sec:turnpike} establishes the turnpike property for the discrete-time stochastic system, leveraging strict dissipativity and cheap control conditions. In Section~\ref{sec:numexp}, we provide numerical experiments to validate our findings. Finally, Section~\ref{sec:conc} concludes the paper and presents possible future interesting developments.  

\section{Stochastic control problem}\label{sec:stochcontrol}
As already mentioned in the introduction, the turnpike phenomenon is a significant principle in optimal control theory. In fact, it establishes that during a specific interior subinterval of a given time horizon, the dynamic control strategy can be approximated with arbitrary precision by the solution of a (generally less complex) static optimization problem. This insight greatly reduces the computational effort typically associated with dynamic control strategies.

In our framework, we start by considering the continuous trajectories of the interacting agents. We denote by $x_k(t) \in \mathbb{R}^d$ the position of agent $k$ at time $t$, with $k = 1, \dots, N$, and the time interval is defined as $[t_0, T]$. The dynamics of the agents are described by the following stochastic differential equation (SDE)
\begin{equation}\label{eq:dynamics}
  dx_k(t) = \left(\frac{1}{N}\sum_{\ell=1}^{N}p(x_k(t),x_\ell(t))(x_\ell(t)-x_k(t)) + u_k(t)\right)dt + \sigma  dW_k(t),
\end{equation}
where the term $\sigma  dW_k(t)$ represents the stochastic noise component, with $W_k(t)$ being standard $d$-dimensional independent Brownian motions. The initial condition for each agent $k$ is given by $x_k(t_0) = x_k^0 \in \mathbb{R}^d$, where $x_k^0$ is a stochastic variable, accounting for randomness in the initial state distribution. The interaction between agents is mediated by a nonlinear interaction kernel $p(x_k(t),x_\ell(t))$ which describes the forces governing their behavior. For further examples and theoretical background, we refer to \cite{attraction}.
The system operates within a well-defined probabilistic framework. Let $(\Omega, \mathcal{F}, \mathbb{P})$ be a complete probability space, where $\Omega$ represents the sample space, $\mathcal{F}$ is the $\sigma$-algebra of measurable events, and $\mathbb{P}$ is the probability measure. The stochastic processes $W_k(t)$, are adapted to the filtration $\mathbb{F} = (\mathcal{F}_t)_{t \geq t_0}$, which represents the natural filtration generated by the Brownian motions. This filtration captures all the information available up to time $t$, ensuring that the evolution of the stochastic processes depends on the past without anticipating future information.

The dynamics of the agents are thus driven by both deterministic controls $u_k(t)$ and the stochastic influences represented by the Brownian motions $W_k(t)$. The control $u_k(t)$ is applied individually to each agent, with the aim of minimizing a cost functional over a finite time horizon.
The control objective is to drive the system towards a desired state $\bar{x}\in \mathbb{R}^d$ while minimizing the associated cost. The latter (to be minimized with respect to the control term) incorporates a tracking term, which measures the deviation of each agent from the desired state $\bar{x}$, and a regularization term to penalize the magnitude of the control. Specifically, the cost functional, which depends on $\bb x= (x_1, \dots, x_N), \ \bb u= (u_1, \dots, u_N)$, is given by
\begin{align}
	\mathcal C(\bb x, \bb u) &= \int_{\Omega} \int_{t_0}^{T} \frac{1}{N} \sum_{k=1}^{N}  \left( \| x_k(t, \omega) - \bar{x} \|^2 + \gamma \| u_k(t) \|^2 \right) dt \, d\mathbb{P}(\omega) \nonumber\\
	&= \mathbb{E} \left[ \frac{1}{N} \int_{t_0}^{T} \sum_{k=1}^{N} \left( \| x_k(t)- \bar{x} \|^2 + \gamma \| u_k (t)\|^2 \right) dt \right].
	\label{eq:cost_functional}
\end{align}
Here, we denote $\| \cdot \|$ as the Euclidean norm on $\mathbb{R}^d$. The parameter $\gamma$ is the regularization parameter penalizing the control effort. Each element $\omega \in \Omega$ corresponds to a specific realization from the sample space, thus describing the inherent randomness of the dynamics.

\subsection{Static control problem}
We now introduce the \emph{static control problem}. In this setting, the static state of agent $k$ is denoted by $\tilde x_k$, and the static control applied to agent $k$ is denoted by $\tilde u_k$. The objective is to minimize the \textit{static} cost functional
\begin{equation}\label{eq:static_cost}
\mathbb{E} \left[ \frac{1}{N} \sum_{k=1}^{N} \left( \, \|\tilde x_k - \bar{x}\|^2 \right)\right]
+ \frac{\gamma}{N} \sum_{k=1}^{N} \|\tilde u_k \|^2,
\end{equation}
with respect to $ \tilde{\bb u}  = (\tilde u_1 , \dots, \tilde u_N )$.
The minimization is subject to the equilibrium condition for each agent $k$, representing the static version of the system's dynamics \eqref{eq:dynamics}. Notice that an equilibrium point for an SDE is a (stochastic) state where the expectation of the dynamics of the system does not change over time. In particular, this means that at the equilibrium the deterministic part of the dynamics (represented by the drift term in~\eqref{eq:dynamics}) is null. Clearly, since the underlying equations are SDEs, even at equilibrium the system exhibits randomness. Therefore, we remark that the equilibrium in our context is a statistical state rather than a deterministic one.
For our specific case, we obtain
\begin{equation}\label{eq:equilibrium_condition}
\frac{1}{N} \sum_{\ell=1}^{N} \mathbb{E} \left( \, p(\tilde x_k , \tilde x_\ell )(\tilde x_\ell  - \tilde x_k ) \, \right) + \tilde u_k   = 0.
\end{equation}
Due to the specific structure of the agent-based dynamics, a solution $\left( \tilde x_k^{*}, \tilde u_k^{*} \right)_{k=1}^{N}$ to the static control problem in \eqref{eq:static_cost}--\eqref{eq:equilibrium_condition} exists but is \textit{not unique}, due to the dependence on the mean and the stochastic nature of the system. Here, we consider the natural candidate solution
\begin{align}\label{eq:stationary_solution}
\tilde x_k^{*} = \bar{x} \quad \text{and} \quad \tilde u_k^{*} = 0.
\end{align}
This implies that, at equilibrium, the static state for each agent coincides with the desired state $\bar{x}$, and no control input is required. Please observe that this specific choice of the static pair in \eqref{eq:stationary_solution} is deterministic.

\section{Time discretization using exponential integrators}\label{sec:timedisc}
As already mentioned in the introduction, we are interested in computing an approximate
solution to the optimal control problem.
The presence of a possibly stiff interaction kernel in the dynamics requires
careful treatment from a time integration point of view. Indeed,
standard explicit integrators (such as the well-known Euler--Maruyama method,
or Runge--Kutta methods in the deterministic case) may be 
subject to severe time step size restrictions
due to the lack of favourable stability properties. To overcome this issue, the research community developed many new techniques and ad-hoc numerical methods.
Here, we focus on the employment of the so-called exponential
integrators~\cite{HO10}, which are \textit{explicit} time marching schemes 
that have proven to perform very well in the stiff regime on many classes of problems 
(see, for instance,
the works~\cite{CEMO21,CC24bis,LPR19,CCEO24,CC24ter,ACCC24} and~\cite{LT13,MT18,LT19} specifically
for the stochastic setting). In particular, for our purposes we will consider a numerical method in the class of the so-called
exponential Rosenbrock schemes~\cite{HOS09,MT18}. The latter are well-suited for
systems of differential equations in which there is no a priori separation of a constant linear stiff part and a non-stiff nonlinear term. This partitioning is required, for instance, to employ the so-called
exponential Runge--Kutta methods~\cite{HO05bis,LT19}.

For convenience of the reader, we report here the main idea at the basis of the
derivation of exponential Rosenbrock integrators in a deterministic autonomous setting (a thorough explanation with full details can be found, e.g., in~\cite{HOS09}). The extension to the stochastic framework is given below.
Let us consider the following general system of ODEs
\begin{equation}\label{eq:ODE}
  \left\{
  \begin{aligned}
  \bb x'(t) &= \bb F(\bb x(t)), \quad t\in(t_0,T], \\
  \bb x(t_0) &= \bb x^0,
  \end{aligned}
  \right.
\end{equation}
and the time discretization $t_{n+1}=t_n+\tau$ for $n=0,\ldots,m-1$. Here $\tau$ denotes the constant time step size and $t_m=T$. Then, we continuously
linearize $\bb F(\bb x(t))$ along the numerical solution $\bb x^n\approx\bb x(t_n)$
as
\begin{equation*}
  \bb F(\bb x(t)) = J_n \bb x(t) + \bb g_n(\bb x(t)),
\end{equation*}
where
\begin{equation}\label{eq:jac}
  J_n = \frac{\partial \bb F}{\partial \bb x}(\bb x^n)
  \quad \text{and} \quad \bb g_n(\bb x(t)) = \bb F(\bb x(t)) - J_n \bb x(t)
\end{equation}
are the Jacobian matrix of the flow evaluated at $\bb x^n$ and the remainder, respectively.
Now, we express the \textit{exact} solution of system~\eqref{eq:ODE} at time $t_{n+1}$ by means of
the variation-of-constants formula
\begin{equation*}
  \bb x (t_{n+1}) = \rme^{\tau J_n}\bb x(t_n) +
  \int_0^{\tau}\rme^{(\tau-s)J_n}\bb g_n (\bb x(t_n + s))ds.
\end{equation*}
A numerical method is then obtained by suitably approximating the integral
term in the formula above. In particular, considering the approximation $g_n (\bb x(t_n + s)) \approx g_n (\bb x(t_n))$,
we get the time marching scheme
\begin{equation}\label{eq:expRB}
  \begin{aligned}
  \bb x^{n+1} &= \rme^{\tau J_n}\bb x^n + \tau \varphi_1(\tau J_n)\bb g_n (\bb x^n)\\
  &= \bb x^n + \tau \varphi_1(\tau J_n)\bb F(\bb x^n).
  \end{aligned}
\end{equation}
This is known as \textit{exponential Rosenbrock--Euler} method.
In formula~\eqref{eq:expRB}, we introduced
the matrix function
$\varphi_1(X) = \int_0^1\rme^{(1-\theta)X}d\theta$ which satisfies the relation
$X\varphi_1(X) = \rme^X-I$. The efficient computation of the matrix exponential and of the exponential-like $\varphi_1$ function
(or their action to vectors) is crucial for an effective employment in practice of the integrator.
For this specific task we refer to the works~\cite{SW09,GRT18,LPR19,AMH11,CCZ23}, among the others.
The exponential Rosenbrock--Euler method is a second-order convergent numerical scheme, it is A-stable by construction (since it integrates exactly linear problems) and it is effective for stiff systems.

A stochastic variant of scheme~\eqref{eq:expRB} can be obtained by using similar reasoning,
see in particular~\cite{MT18}. In fact, a numerical solution to the system of SDEs
\begin{equation*}
  \left\{
  \begin{aligned}
    d\bb x(t) &= \bb F(\bb x(t)) dt + \bb H(\bb x(t))d \bb W(t), \quad t\in(t_0,T], \\
    \bb x(t_0) &= \bb x^0,
  \end{aligned}
  \right.
\end{equation*}
where $d \bb W(t)$ is the term accounting for the Brownian motions,
can be obtained by the so-called \textit{stochastic exponential Rosenbrock--Euler}
method
\begin{equation}\label{eq:serb}
  \bb x^{n+1} = \bb x^n + \tau\varphi_1(\tau J_n)(\bb F(\bb x^n) + \bb H(\bb x^n)(J_n d \bb W^n)) + \bb H(\bb x^n) d \bb W^n.
\end{equation}
Here, $d\bb W^n = \bb W^{n+1} - \bb W^n$ and $J_n$ is defined in~\eqref{eq:jac}.
The stochastic Rosenbrock--Euler method is a first order scheme (in strong sense)
which is suitable for stiff mmetrstems of SDEs. Notice that if we remove the Brownian
motion part from the system, the integrators reduces to the standard 
exponential Rosenbrock--Euler method~\eqref{eq:expRB}.
\begin{remark}
Scheme~\eqref{eq:serb} generalizes straightforwardly to non-autonomous problems in the form
\begin{equation*}
  \left\{
  \begin{aligned}
    d\bb x(t) &= \bb F(t,\bb x(t)) dt + \bb H(\bb x(t))d \bb W(t), \quad t\in(t_0,T], \\
    \bb x(t_0) &= \bb x^0.
  \end{aligned}
  \right.
\end{equation*}
In fact, by equivalently rewriting the system in autonomous formulation (i.e., introducing an additional variable for the time), applying scheme~\eqref{eq:serb}, and employing relations of the $\varphi_1$ function we get the time marching method
\begin{equation}\label{eq:serbaut}
  \bb x^{n+1} = \bb x^n + \tau\varphi_1(\tau J_n)(\bb F(\bb x^n) + \bb H(\bb x^n)(J_n d \bb W^n)) + \bb H(\bb x^n) d \bb W^n
  +\tau^2\varphi_2(\tau J_n)\bb v_n.
\end{equation}
Here, we define
\begin{equation*}
J_n = \frac{\partial \bb F}{\partial \bb x}(t_n,\bb x^n), \quad
\bb v_n=\frac{\partial \bb F}{\partial t}(t_n,\bb x^n),
\end{equation*}
and
$\varphi_2(X)=\int_0^1\rme^{(1-\theta)X}\theta d\theta$ still belongs to the
class of the so-called $\varphi$-functions.
\end{remark}

Recall that, in our case, we are interested in time integrating
\begin{equation*}
  dx_k(t) = \left(\frac{1}{N}\left(\sum_{\ell=1}^{N}p(x_k(t),x_\ell(t))(x_\ell(t)-x_k(t))\right) + u_k(t)\right)dt + \sigma  dW_k(t),
\end{equation*}
for each agent $k=1,\ldots,N$. In the following, we assume for simplicity of notation $d=1$, i.e., $x_k(t)\in \RR$. This assumption will also be used in the numerical examples later and will remain in effect throughout the entire paper. Then, we can write an equivalent matrix formulation of the problem as
\begin{equation}\label{eq:dyn_matrix}
  \begin{aligned}
  d\bb x (t) &= \left(\frac{1}{N}P(\bb x(t))\bb x(t) - \frac{1}{N}\bb x(t)\bb s(\bb x(t)) + \bb u(t)\right)dt +\sigma d\bb W (t) \\
             &= \bb F(t,\bb x(t)) dt + \sigma d\bb W (t),
  \end{aligned}
\end{equation}
where $\bb x(t)=(x_k(t))\in\RR^{N}$ is the vector containing the position of each
agent $k$ at time $t$, $P(\bb x(t))=(p_{k\ell}(t))\in\RR^{N\times N}$
(with $p_{k\ell}(t)=p(x_k(t),x_\ell(t))$) is the matrix representing the interactions
among agents, $\bb s(x(t)) = (s_k(t))\in \RR^N$ (with $s_k(t)=\sum_{\ell}p(x_k(t),x_\ell(t))$), 
$\bb u(t) = (u_k(t))\in \RR^N$ is the vector containing the control
for each agent $k$, and $d\bb W(t)=(dW_k(t))\in\RR^N$ is the vector of the Brownian
motions.
Then, employing~\eqref{eq:serbaut} we get the time marching
\begin{equation}\label{eq:exp_int_discr}
  \bb x^{n+1} = \bb x^n
  +\tau\varphi_1(\tau J_n)(\bb F(\bb x^n) + \sigma J_n d \bb W^n)
  + \sigma d \bb W^n + \tau^2\varphi_2(\tau J_n)\bb v_n,
\end{equation}
where the element $(k,\ell)$ of the matrix $J_n$ is given by
\begin{equation*}
  j_{k\ell} =
  \begin{cases}
  \frac{1}{N}\sum\limits_{i\neq k}\left(\frac{\partial}{\partial x}p(x,y)\lvert_{(\bb x^n_k, \bb x^n_i)} (\bb x^n_i - \bb x^n_k) - p(\bb x^n_k,\bb x^n_i)\right) & \text{if } k = \ell \\
  \frac{1}{N}\left(\frac{\partial}{\partial y}p(x,y)\lvert_{(\bb x^n_k, \bb x^n_\ell)} (\bb x^n_\ell - \bb x^n_k) + p(\bb x^n_k,\bb x^n_\ell)\right)& \text{otherwise}
  \end{cases}
\end{equation*}
while the $k$-th element of $\bb v_n$ is $u'_k(t_n)$.

\section{The turnpike property for the time-discrete stochastic control problem}\label{sec:turnpike}
In this section, we focus on the turnpike property with \emph{interior decay} in the context of stochastic control problem for interacting agents, discretized in time using exponential integrators (introduced in Section \ref{sec:timedisc}).

In order to study the turnpike property in this setting, we first consider the discretized dynamics in formula~\eqref{eq:exp_int_discr}. Then, to complete the formulation, we must also discretize the cost functional~\eqref{eq:cost_functional}. To this aim, we employ a first-order quadrature rule (rectangle left point, in particular).
Therefore, the fully discrete optimal control problem, which we denote by
$\mathcal{Q}(\tau, t_0, T, \bb{x}^0)$, is formulated as
\begin{equation}
\min_{\bb u} \mathcal C^{\tau}(\bb x, \bb u):= \mathbb{E} \left[  \frac{\tau}{N} \sum_{n=0}^{m-1}\sum_{k=1}^{N}  \left( | x^n_k - \bar{x} |^2 + \gamma | u^n_k |^2 \right) \right],
\label{eq:discrete_cost_functional}
\end{equation}
where we recall that $x_k^n \in \mathbb{R}$ approximates $x_k(t_n)$ (i.e., the position of the $k$-th agent at the $n$-th time step) through the dynamics, and $u_k^n \in \mathbb{R}$ is the control applied at the same agent and step.
Additionally, we define the \emph{running cost} at each time step, which describes the contribution of the state and control at time step $n$ to the overall cost functional. This is given by
\begin{equation}\label{eq:running_cost}
    c(\bb x^n, \bb u^n) := \mathbb{E} \left[ \frac{1}{N} \sum_{k=1}^{N} \left( | x^n_k - \bar{x} |^2 + \gamma | u^n_k |^2  \right) \right].
\end{equation}
The running cost $c$ measures the deviation of the system from the desired state $\bar{x}$ and penalizes the control effort at each discrete time step $n$. In these terms, the minimization in formula~\eqref{eq:discrete_cost_functional} can be written as
\begin{equation*}
\min_{\bb u} \mathcal C^{\tau}(\bb x, \bb u) = \tau\sum_{n=0}^{m-1}c(\bb x^n, \bb u^n).
\end{equation*}
Notice that the time-discrete optimization problem \eqref{eq:discrete_cost_functional} depends also on the initial time $t_0$, the terminal time $T$, and the discretization step size $\tau$. In the time-discrete setting, we assume that these parameters are chosen so that the number of time steps $m$ is given by
\begin{equation}\label{eq:timestep}
    m = \frac{T - t_0}{\tau},
\end{equation}
where $m \in \mathbb{N}$ represents the total number of time steps between $t_0$ and $T$. The discrete-time optimization problem thus consists of minimizing the cost functional over $m$ time steps.
The optimal value of the time-discrete problem is denoted by
\begin{equation*}
    \mathcal V^{(\tau, t_0, T)}(\bb x^0) := \min_{\bb u} \mathcal{C}^{\tau}(\bb x, \bb u),
\end{equation*}
where $\bb x^0$ is the initial state of the system, representing the positions of all agents at the initial time $t_0$.

This formulation captures the time-discrete version of the stochastic control problem, enabling us to apply the turnpike property analysis in the time-discrete setting.
In the next steps, we will examine the conditions under which the turnpike property holds for this time-discrete stochastic control problem. Specifically, we will investigate the proximity of the optimal dynamic trajectory to the corresponding static solution, particularly in the interior of the time horizon. In the upcoming analysis, given a state vector $\bb{x} = ({x}_1, {x}_2, \dots, {x}_N)^{\sf T}$ in $\mathbb{R}^{N}$, we consider the norm
\begin{equation*}
\|\bb{x}\|= \sqrt{\sum_{k=1}^{N} |{x}_k|^2}.
\end{equation*}  
This corresponds to the standard Euclidean norm in $\mathbb{R}^N$ and will serve as our primary measure of magnitude for the state vector.

\subsection{The strict dissipativity property}
A key requirement for establishing the turnpike property is the concept of \emph{strict dissipativity} of the cost functional. This notion, originally introduced in~\cite{willems1972dissipative}, plays a fundamental role in optimal control theory. Dissipativity provides a connection between the system's dynamics and its cost structure, ensuring that the system’s behavior converges to an optimal steady-state over time. While widely used in continuous-time control systems, here we extend it to the time-discrete setting, focusing on the optimal control problem $\mathcal{Q}(\tau, t_0, T, \bb{x}^0)$ formulated for interacting agents. 

To formalize strict dissipativity in this context, we introduce the key components. The \emph{supply rate function} $\eta(\bb{x}^n, \bb{u}^n)$ measures the deviation of the current state-control pair from the optimal static pair $(\tilde{\bb{x}} , \tilde{\bb{u}} )$. The \emph{storage function} $S: \mathbb{R}^{N} \to \mathbb{R}$, bounded from below acts as an energy-like function capturing the system’s stored cost. The \emph{dissipation function} $\alpha: [0, \infty) \to [0, \infty)$, continuous and monotone increasing with $\alpha(0) = 0$, quantifies the rate at which deviations from the optimal steady-state are penalized.
In our setting, the supply rate function is chosen as
\begin{equation*}
    \eta(\bb{x}^n, \bb{u}^n) = c(\bb{x}^n, \bb{u}^n) - c(\tilde{\bb{x}}, \tilde{\bb{u}}),
\end{equation*}
where $c(\cdot,\cdot)$ is the running cost defined in \eqref{eq:running_cost}. Since $(\tilde{\bb{x}}, \tilde{\bb{u}})$ is the optimal static pair minimizing the running cost, we have $c(\tilde{\bb{x}}, \tilde{\bb{u}}) = 0$. Thus, the supply function simplifies to $\eta(\bb{x}^n, \bb{u}^n) = c(\bb{x}^n, \bb{u}^n)$, which directly quantifies how far the system is from the optimal steady-state.
To ensure strict dissipativity, we require the existence of a storage function $S(\cdot)$ and a dissipation function $\alpha(\cdot)$ satisfying a key inequality. Specifically, defining $\alpha(y) = \frac{\gamma}{2N} y^2$ and assuming a control penalization parameter $\gamma \in (0,1]$, we ensure that larger deviations from the optimal steady-state incur a greater cost, promoting convergence.

We now present the complete statement of strict dissipativity.
\begin{lemma}\label{lem:diss}
Let the control penalization parameter $\gamma \in (0,1]$. The time-discrete optimal control problem $\mathcal{Q}(\tau, t_0, T, \bb{x}^0)$ is strictly dissipative with respect to the supply rate function $\eta(\bb{x}^n, \bb{u}^n)$. That is, there exists a storage function $S: \mathbb{R}^{N} \to \mathbb{R}$, bounded from below, and a continuous, monotone increasing function $\alpha: [0, \infty) \to [0, \infty)$ with $\alpha(0) = 0$, such that for all state-control pairs $(\bb{x}^n, \bb{u}^n)$ satisfying the discrete dynamics \eqref{eq:exp_int_discr}, the following inequality holds in expectation over the noise
\begin{equation}\label{eq:dissipativity_ineq}
    \mathbb{E}\left[S(\bb{x}^n) + \tau \, \eta(\bb{x}^n, \bb{u}^n)\right] \geq \mathbb{E}\left[S(\bb{x}^{n+1}) + \tau \, \alpha\left(\|\bb{x}^n - \tilde{\bb{x}}\| + \|\bb{u}^n - \tilde{\bb{u}}\| \right)\right].
\end{equation}
\end{lemma}
\begin{proof}
To establish strict dissipativity, we recall that $c(\tilde{\bb{x}}, \tilde{\bb{u}}) = 0$ and the choices $\eta(\bb{x}^n, \bb{u}^n) = c(\bb{x}^n, \bb{u}^n)$ and $\alpha(y) = \frac{\gamma}{2N} y^2$. Then,
using these definitions and the properties of the optimal static pair, we derive the following sequence of inequalities
\begin{align*}
    \alpha\left( \|\bb{x}^n - \tilde{\bb{x}}\| + \|\bb{u}^n - \tilde{\bb{u}}\| \right)
    &= \frac{\gamma}{2N} \left( \|\bb{x}^n - \tilde{\bb{x}}\| + \|\bb{u}^n - \tilde{\bb{u}}\| \right)^2 \\
    &\leq \frac{\gamma}{N} \left( \|\bb{x}^n - \tilde{\bb{x}}\|^2 + \|\bb{u}^n - \tilde{\bb{u}}\|^2 \right) \\
    &= \frac{\gamma}{N} \sum_{k=1}^N \left( |x_k^n - \tilde x_k|^2 + |u_k^n - \tilde u_k|^2 \right) \\
    &= \frac{\gamma}{N} \sum_{k=1}^N \left( |x_k^n - \bar{x}|^2 + |u_k^n|^2 \right) \\
    &\leq \frac{1}{N} \sum_{k=1}^N \left( |x_k^n - \bar{x}|^2 + \gamma |u_k^n|^2 \right) \\
    &= c(\bb{x}^n, \bb{u}^n).
\end{align*}
Here we also used the fact that $\gamma \leq 1$. Finally, since the noise in the system must be taken into account, we incorporate expectations on both sides of the dissipativity inequality. This ensures that the inequality holds on average, accounting for the stochastic nature of the problem. Thus, multiplying both sides by $\tau$, and choosing $S(\cdot)=0$, we can write  the strict dissipativity condition in the mean sense
\begin{equation*}
    \mathbb{E}\left[S(\bb{x}^n) + \tau \, \eta(\bb{x}^n, \bb{u}^n)\right] \geq \mathbb{E}\left[S(\bb{x}^{n+1}) + \tau \, \alpha\left(\|\bb{x}^n - \tilde{\bb{x}}\| + \|\bb{u}^n - \tilde{\bb{u}}\| \right)\right].
\end{equation*}
This concludes the proof.
\end{proof}
Remark that, in our setting, we obtain that the time-discrete optimal control problem $\mathcal{Q}(\tau, t_0, T, \bb{x}^0)$ satisfies the strict dissipativity property~\eqref{eq:dissipativity_ineq} with $S(\cdot)=0$ and $\alpha(y)=\frac{\gamma}{2N}y^2$. This in particular means that deviations from the static optimal solution are penalized in a way that guarantees convergence to the static optimal pair $(\tilde{\bb{x}}, \tilde{\bb{u}})$ as the time horizon progresses.

\subsection{The cheap control condition}
In the pursuit of establishing the turnpike property, we identify the cheap control condition as the second essential ingredient, following the dissipativity condition we explored earlier. While the dissipativity condition ensures that the system's behavior remains well-defined and controllable, the cheap control assumption introduces a crucial bound that relates to the distance between the initial state and the steady-state solution of the time-discrete problem. The combination of these two properties will ultimately facilitate our analysis of the turnpike phenomenon.

The cheap control condition asserts that the cost required to steer the system from an initial state to the steady-state solution remains bounded by a function of the initial deviation and a term related to storage function differences. This condition provides a measure of how efficiently the system can be controlled and ensures that the overall control effort remains manageable.
To properly set up the framework, we introduce the following assumption.
\begin{ass}\label{ass:bounded}
We assume that the interaction kernel $p(\cdot, \cdot)$ in \eqref{eq:dynamics} is bounded, i.e., there exists a constant $M_p > 0$ such that 
\[
|p(x, y)| \leq M_p, \quad \text{for all } x, y \in \mathbb{R}.
\]
\end{ass}
The boundedness of the kernel ensures that the interaction effects remain controlled and do not grow unreasonably large. This assumption is important for guaranteeing stability in the system's dynamics and facilitating the analysis of control strategies, even in cases where the interactions may be asymmetric.
We now present the formal statement of the cheap control condition.
\begin{lemma}\label{lem:cheap}
Under Assumption \ref{ass:bounded}, the optimal control problem $\mathcal{Q}(\tau, t_{0}, T, \bb{x}^{0})$ is cheaply controllable, that is, there exist constants $C_0 > 0$ and $\varepsilon_0 \geq 0$ such that for all initial times $t_{0}$, all initial states $\bb{x}^{0}$, and for all terminal times $T > t_{0}$, the following inequality holds
\begin{equation}\label{eq:cheap}
    \mathcal{V}^{(\tau, t_{0}, T)}(\bb{x}^{0}) \leq \mathbb{E} \left[ C_{0} \, \alpha(\|\bb{x}^{0} - \tilde{\bb{x}}\|) + \varepsilon_{0} + S(\bb{x}^{m}) - S(\bb{x}^{0}) \right].
\end{equation}
\end{lemma}
\begin{proof}
We prove the cheap controllability condition with $\varepsilon_{0} = 0$. As in the proof of Lemma \ref{lem:diss}, we select $\alpha(y) = \frac{\gamma}{2N} y^2$ and $S(\cdot) = 0$.
Then, with these choices, to establish the validity of inequality \eqref{eq:cheap}, we need to prove that there exists a constant ${C}_{0}$ such that for all initial times $t_{0}$, terminal times $T$, and initial states $\bb{x}^{0}$, the following relationship holds
\begin{equation*}
    \mathcal{V}^{(\tau, t_{0}, T)}(\bb{x}^{0}) \leq {C}_{0} \, \mathbb{E}\left[{\frac{\gamma}{2N}} \|\bb{x}^{0} - \tilde{\bb{x}}\|^{2}\right].
\end{equation*}
To prove it, we can leverage a stabilizing feedback control law, similar to the one discussed in \cite{main}, which accommodates the stochastic dynamics of the system. This feedback mechanism is pivotal as it facilitates the exponential decay of the running cost $c$ defined in \eqref{eq:running_cost}.
Let us consider a parameter $\beta>0$.
For each agent $k = 1, \ldots, N$ and time step $n = 0, \ldots, m$, we define the \textit{cheap} control input as follows
\begin{equation}\label{eq:cheap_ctr}
    u_{k}^{n} = \mathbb{E}\left[ u_{k}^{n}(\omega) \right] = \mathbb{E}\left[ \beta(\bar{x} - x_{k}^{n}) - \frac{1}{N} \sum_{\ell=1}^{N} p(x_{k}^{n}, x_{\ell}^{n})(x_{\ell}^{n} - x_{k}^{n}) \right].
\end{equation}
The corresponding state $\bb x^n$ evolves from any initial state $\bb x^0$ in accordance with the scheme described by~\eqref{eq:exp_int_discr}, using the mean control strategy in~\eqref{eq:cheap_ctr}. Observe that at this stage we use the autonomous version of scheme~\eqref{eq:exp_int_discr} (in particular $\bb v_n = 0$), because the control in \eqref{eq:cheap_ctr} does not depend on time.
If we substitute the control law \eqref{eq:cheap_ctr} into the stochastic dynamics \eqref{eq:dynamics} and we take the expectation on both sides (since our focus is on the mean evolution of the agents' states), after simple calculations we obtain the following
\begin{equation}\label{eq:dyn_mean}
  \frac{d}{dt} \mathbb{E}\left[ x_k(t) \right] = \beta \, \bar{x} - \beta \, \mathbb{E}\left[ x_{k}(t)) \right].
\end{equation}
Applying now the \textit{stochastic exponential Rosenbrock--Euler} method we get
\begin{equation}\label{eq:expint_mean}
  \hat{\bb x}^{n+1} = \hat{\bb x}^n
  +\tau \, \varphi_1(-\beta \tau) \, \beta \, ( \bar{\bb x} - \hat{\bb x}^{n}),
\end{equation}
where we defined $\hat{\bb x}^n := \mathbb{E}\left[ \bb x^n \right]$.
Notice that in~\eqref{eq:dyn_mean} where we are dealing with the mean dynamics, the \textit{stochastic exponential Rosenbrock--Euler} method effectively coincides with the standard \textit{exponential Rosenbrock--Euler} method. This is because, in the evolution of the mean, the uncertainty introduced by the noise term disappears, as the expectation of the Brownian motion is zero. Then, we are left with a purely deterministic evolution for the mean. Furthermore, since the exponential Rosenbrock--Euler method is exact for linear equations (and \eqref{eq:dyn_mean} is linear), the numerical scheme \eqref{eq:expint_mean} is not just an approximation, but actually represents the exact solution to the mean dynamics. This is not true, e.g., if we employ a standard explicit method as explicit Euler.

To measure how close the system is to the static state, we introduce the quadratic Lyapunov function
\begin{equation*}
L^n = \| \hat{\bb x}^n - \tilde{\bb{x}}\|^{2}.
\end{equation*}
Using~\eqref{eq:expint_mean}, the evolution of this function is governed by the following recurrence
\begin{equation*}
L^{n+1} = \|\hat{\bb x}^{n+1} - \tilde{\bb{x}}\|^{2} = \|\hat{\bb x}^{n} - \tilde{\bb{x}} - \tau \beta   \, \varphi_1(-\beta \tau) (\hat{\bb x}^{n} - \tilde{\bb{x}})\|^{2}.
\end{equation*}
Expanding this expression, we get
\begin{equation*}
L^{n+1} = (1 - \tau \beta   \, \varphi_1(-\beta \tau) )^2 \|\hat{\bb x}^{n} - \tilde{\bb{x}}\|^{2} = \rme^{-2\beta\tau} L^n.
\end{equation*}
This shows that the Lyapunov function decays exponentially at each time step. Thus, for all $n = 0, \dots, m$,
\begin{equation*}
L^n = \rme^{-2n\beta\tau} L^0 = \rme^{-2n\beta\tau} \|\hat{\bb{x}}^{0} - \tilde{\bb{x}}\|^{2}.
\end{equation*}
Now, starting from the \textit{cheap} control \eqref{eq:cheap_ctr}, we apply the triangle inequality and we use the boundedness of \( p \) in Assumption \ref{ass:bounded}:

\begin{align*}
|u_{k}^{n}|
&\leq \beta | \bar{x} - \hat{x}_{k}^{n} | + \frac{M_p}{N} \sum_{\ell=1}^{N} \left( |\hat x_{\ell}^{n} -  \bar{x}| + |\hat x_{k}^{n} - \bar{x}| \right) \notag \\
&= (\beta + M_p) | \bar{x} - \hat{x}_{k}^{n} | + \frac{M_p}{N} \sum_{\ell=1}^{N} |\hat x_{\ell}^{n} -  \bar{x}|.
\end{align*}
We remind that we indicate $\hat x_k^n = \mathbb E [x_k^n]$. Next, squaring both sides and applying the inequality \( |a + b|^2 \leq 2(|a|^2 + |b|^2) \), as well as the Cauchy-Schwarz inequality as \( \left( \sum_{i=1}^N a_i \right)^2 \leq N \sum_{i=1}^N  a_i^2 \), we get:

\begin{align*}
|u_{k}^{n}|^2 
&\leq 2 \left[ (\beta + M_p)^2 | \bar{x} - \hat{x}_{k}^{n} |^2 + \frac{M_p^2}{N} \sum_{\ell=1}^{N} |\hat x_{\ell}^{n} -  \bar{x}|^2 \right].
\end{align*}
Summing over \( k \)
we obtain:

\begin{align*}
\| \bb u^n \|^2 
&\leq 2 \left[ (\beta + M_p)^2 \sum_{k=1}^{N} | \bar{x} - \hat{x}_{k}^{n} |^2 + M_p^2 \sum_{\ell=1}^{N} |\hat x_{\ell}^{n} -  \bar{x}|^2 \right]  \\
&= 2 (\beta + M_p)^2 \left\lVert \hat{\bb x}^{n} - \tilde{\bb{x}}\right\lVert^2 + 2 M_p^2 \left\lVert \hat{\bb x}^{n} - \tilde{\bb{x}}\right\lVert^2  \\
&= 2 \left( (\beta + M_p)^2 + M_p^2 \right) \left\lVert \hat{\bb x}^{n} - \tilde{\bb{x}}\right\lVert^2.
\end{align*}
Defining the positive constant $\beta_p := 2 \left( (\beta + M_p)^2 + M_p^2 \right)$, we obtain
\begin{align*}
    \|\bb{u}^n\|^2 \leq \beta_p \left\lVert \hat{\bb x}^{n} - \tilde{\bb{x}}\right\lVert^2.
\end{align*}
We can then bound the total cost $\mathcal{V}^{(\tau, t_0, T)}(\bb{x}^0)$
\begin{align*}
    \mathcal{V}^{(\tau, t_0, T)}(\bb{x}^0) &= \min_{u} \mathbb{E} \left[  \frac{\tau}{N} \sum_{n=0}^{m-1}\sum_{k=1}^{N}  \left( | x^n_k - \bar{x} |^2 + \gamma | u^n_k |^2 \right) \right]
    \\
    &\leq \sum_{n=0}^{m-1} \frac{\tau}{N} \, \mathbb{E} \left[ \|\bb{x}^{n} - \tilde{\bb{x}}\|^{2} + \gamma \|\bb{u}^{n}\|^{2} \right]
    \\
    &\leq \sum_{n=0}^{m-1} \frac{\tau}{N} \left( 1 + \gamma \, \beta_p \right) L^n.
\end{align*}
Substituting the exponential decay of $L^n$, we obtain
\begin{equation*}
\mathcal{V}^{(\tau, t_0, T)}(\bb{x}^0) \leq \sum_{n=0}^{m-1} \frac{\tau}{N} \left( 1 + \gamma \, \beta_p \right) \rme^{-2n\beta\tau} L^0.
\end{equation*}
Since the series $\sum_{n=0}^{\infty} \rme^{-2n\beta\tau}$ converges, we can write
\begin{equation*}
\mathcal{V}^{(\tau, t_0, T)}(\bb{x}^0) \leq \widetilde{C} \, \frac{1}{N}\|\hat{\bb{x}}^{0} - \tilde{\bb{x}}\|^{2} \leq 
 {C}_{0} \, \mathbb{E}\left[{\frac{\gamma}{2N}} \|\bb{x}^{0} - \tilde{\bb{x}}\|^{2}\right],
 \end{equation*}
where 
\begin{equation*}
C_0 = \frac{2}{\gamma}\widetilde{C} = \frac{2 \tau \, (1 + \gamma \, \beta_p )}{\gamma(1 - \rme^{-2\beta\tau}) }>0
\end{equation*}
is a constant that depends on the control parameters. This concludes the proof.
\end{proof}

\subsection{The turnpike property with interior decay}
In this section, we establish the turnpike property with interior decay. We recall that this concept highlights that the optimal solution of a \emph{dynamic} optimal control problem remains close to the corresponding \emph{static} optimal solution for most of the time interval, provided the time horizon is sufficiently large.

Our main final result is the following (Theorem \ref{thm:turnpike_stochastic}), which will be proven by leveraging the dissipativity inequality established earlier and the property of cheap controllability. 
\begin{theorem}\label{thm:turnpike_stochastic}
The solution to the optimal control problem \(\mathcal{Q}(\tau, t_0, T, \bb{x}^0)\), as described in \eqref{eq:discrete_cost_functional}, exhibits the turnpike property with a decay towards the interior. Specifically, there exist constants \(C_1 > 0\), \(\lambda \in (0,1)\), and a non-negative, monotonically increasing function \(\alpha\) with \(\alpha(0) = 0\), such that for any \(T > t_0\) and any \(m \in \mathbb{N}\) satisfying \eqref{eq:timestep} the following holds
\begin{equation*}
\mathbb{E}\left[\sum_{n=\left \lfloor (1-\lambda)m\right \rfloor}^{m-1} \tau \, \alpha\left(\|\bb{x}^n - \tilde {\bb{x}}\| + \|\bb{u}^n - \tilde{\bb{u}}\|\right)\right] \leq C_1 \, \mathbb{E} \left[ \alpha(\|\bb{x}^0 - \tilde{\bb{x}}\|) \right].
\end{equation*}
\end{theorem}
\begin{proof}
We recall that, under the assumptions, the problem  \(\mathcal{Q}(\tau, t_0, T, \bb{x}^0)\) fulfills the strict dissipativity inequality \eqref{eq:dissipativity_ineq} and it is cheaply controllable \eqref{eq:cheap} for $S(\cdot)=\varepsilon_{0}=0$ and $\alpha(y) = \frac{\gamma}{2 N}y^2.$
Let now \((\bb{x}^*, {\bb{u}}^*)\) represent the optimal state and control for the problem \(\mathcal{Q}(\tau, t_0, T, \bb{x}^0)\). Using the dissipativity inequality and the property of cheap control, we derive
\begin{equation}
\begin{aligned}
\mathbb{E} &\left[ \tau \sum_{n=0}^{m-1} \alpha \left(\|\bb{x}^{n,*} - \tilde{\bb{x}}\| + \|\bb{u}^{n,*} - \tilde{\bb{u}}\|\right) \right] \\
&\le \mathbb{E} \left[ \tau \sum_{n=0}^{m-1} c(\bb{x}^{n,*}, \bb{u}^{n,*}) \right] \\
&= \mathcal{V}^{(\tau, t_0, T)}(\bb{x}^0) \\
&\le \mathbb{E} \left[ C_0 \, \alpha(\|\bb{x}^0 - \tilde{\bb{x}}\|) \right].
\end{aligned}
\label{eq:turnpike_ineq}
\end{equation}
Define
\begin{equation*}
    s = \left \lfloor (1 - \lambda)m \right \rfloor,
\end{equation*}
with $\lambda\in(0,1)$. Then, there exists an index \(s^* \in \{0, 1, \dots, s - 1\}\) for which the following inequality is satisfied
\begin{equation}
   \mathbb{E} \left[ \tau \, \alpha\left(\|\bb{x}^{s^*,*} - \tilde{\bb{x}}\| + \|\bb{u}^{s^*,*} - \tilde{\bb{u}}\|\right) \right] \le \frac{1}{s} \, \mathbb{E} \left[ C_0 \, \alpha(\|\bb{x}^0 - \tilde{\bb{x}}\|) \right].
\label{eq:turnpike_ineq2}
\end{equation}
Indeed, by contradiction assume that for all \(n \in \{0, 1, \dots, s - 1\}\), the inequality
\begin{equation*}
    \mathbb{E} \left[ \tau \, \alpha\left(\|\bb{x}^{n,*} - \tilde{\bb{x}}\| + \|\bb{u}^{n,*} - \tilde{\bb{u}}\|\right) \right] > \frac{1}{s} \, \mathbb{E} \left[ C_0 \, \alpha(\|\bb{x}^0 - \tilde{\bb{x}}\|) \right]
\end{equation*}
is valid. Then, it would imply
\begin{equation*}
\mathbb{E} \left[ \tau \sum_{n=0}^{s - 1} \alpha\left(\|\bb{x}^{n,*} - \tilde{\bb{x}}\| + \|\bb{u}^{n,*} - \tilde{\bb{u}}\|\right) \right] > \mathbb{E} \left[ C_0 \, \alpha(\|\bb{x}^0 - \tilde{\bb{x}}\|) \right],
\end{equation*}
contradicting \eqref{eq:turnpike_ineq}.
Therefore, we get
\begin{equation}
\begin{aligned}
\mathbb{E}&\left[\sum_{n=s}^{m-1} \tau  \, \alpha\left(\|\bb{x}^{n,*} - \tilde{\bb{x}}\| + \|\bb{u}^{n,*} - \tilde{\bb{u}}\|\right)\right] \\
&\le \mathbb{E}\left[\sum_{n=s^*}^{m-1} \tau \, \alpha\left(\|\bb{x}^{n,*} - \tilde{\bb{x}}\| + \|\bb{u}^{n,*} - \tilde{\bb{u}}\|\right)\right] \\
&\le \mathbb{E}\left[\sum_{n=s^*}^{m-1} \tau \, c(\bb{x}^{n,*}, \bb{u}^{n,*})\right] = \mathcal{V}^{(\tau, t_{s^*}, T)}(\bb{x}^{s^*,*}).
\end{aligned}
\label{eq:turnpike_ineq3}
\end{equation}
The value \(\mathcal{V}^{(\tau, t_{s^*}, T)}(\bb{x}^{s^*,*})\) is the optimal cost for the problem \(\mathcal{Q}(\tau, t_{s^*}, T, \bb{x}^{s^*,*})\) initialized at \(t_{s^*}\) with state \(\bb{x}^{s^*,*}\). Using the cheap control result (Lemma~\ref{lem:cheap}), we have
\begin{equation}
    \mathcal{V}^{(\tau, t_{s^*}, T)}(\bb{x}^{s^*,*}) \le \mathbb{E} \left[ C_0 \, \alpha(\|\bb{x}^{s^*,*} - \tilde{\bb{x}}\|) \right].
\label{eq:turnpike_ineq4}
\end{equation}
Combining \eqref{eq:turnpike_ineq2} with \eqref{eq:turnpike_ineq4}, exploiting that $\alpha(\cdot)$ is monotone increasing, we find
\begin{equation*}
    \mathcal{V}^{(\tau, t_{s^*}, T)}(\bb{x}^{s^*,*}) \le \frac{1}{\tau s}  \, \mathbb{E} \left[ C_0^2 \, \alpha(\|\bb{x}^0 - \tilde{\bb{x}}\|) \right].
\end{equation*}
Finally, substituting this into \eqref{eq:turnpike_ineq3} gives
\begin{equation*}
    \mathbb{E}\left[\sum_{n=s}^{m-1} \tau \alpha\left(\|\bb{x}^{n,*} - \tilde{\bb{x}}\| + \|\bb{u}^{n,*} - \tilde{\bb{u}}\|\right)\right] \le \frac{C_0^2}{\tau s} \,  \mathbb{E} \left[ \alpha(\|\bb{x}^0 - \tilde{\bb{x}}\|) \right].
\end{equation*}
We now define the \(\tau\)-dependent constant \(C_1 = C_1(\tau)\) as
\begin{equation*}
    C_1 := \frac{C_0^2}{\tau \left \lfloor (1 - \lambda) m \right \rfloor}.
\end{equation*}
This completes the proof.
\end{proof}

\subsection{Estimation of the turnpike time \texorpdfstring{$\overline{t}$}{tbar}}\label{sec:estimate}
The turnpike time $\bar{t} = \bar{n} \tau$ corresponds to the time after which the cheap control can be switched off and replaced with the static control. 
To determine the turnpike time $\bar{t}$ theoretically, we leverage the decay of the Lyapunov functional. Recall that the Lyapunov functional is defined as 
\begin{equation*}
L^n = \|\hat{\bb{x}}^n - \tilde{\bb{x}}\|^2,
\end{equation*}
and its evolution at each time step is given by
\begin{equation*}
L^{n+1} = \rme^{-2\beta\tau} L^n.
\end{equation*}
Thus, for all \(n = 0, \dots, m\), the decay can be expressed explicitly as 
\begin{equation*}
L^n = \rme^{-2n\beta\tau} L^0 = \rme^{-2n\beta\tau} \|\hat{\bb{x}}^0 - \tilde{\bb{x}}\|^2.
\end{equation*}
To estimate $\bar{t}$, we impose a threshold $\delta > 0$ that quantifies how close the system state should be to the target state $\tilde{\bb{x}}$. Specifically, we require
\begin{equation*}
L^{\bar{n}} \leq \delta,
\end{equation*}
which, using the decay expression, translates to 
\begin{equation*}
\rme^{-2\bar{n}\beta\tau} \|\hat{\bb{x}}^0 - \tilde{\bb{x}}\|^2 \leq \delta.
\end{equation*}
Dividing by $\|\hat{\bb{x}}^0 - \tilde{\bb{x}}\|^2$ and taking the natural logarithm, we obtain
\begin{equation*}
-2\beta\bar{t}\leq \ln\left(\frac{\delta}{\|\hat{\bb{x}}^0 - \tilde{\bb{x}}\|^2}\right).
\end{equation*}
Therefore, we conclude that the turnpike time may be taken as
\begin{equation}\label{eq:tptime}
\bar{t} \geq \frac{\ln\left(\delta/\|\hat{\bb{x}}^0 - \tilde{\bb{x}}\|^2\right)}{-2\beta}.
\end{equation}
This theoretical estimate provides a systematic way to determine $\bar{t}$ based on the parameters $\beta$, the expectation of the initial state $\hat{\bb{x}}^0$, the target state $\tilde{\bb{x}}$, and the desired threshold $\delta$. By selecting an appropriate $\delta$, one can identify the time at which the cheap control can be replaced by the static control.

\section{Numerical experiments}\label{sec:numexp}
We now present several numerical experiments that demonstrate the theoretical findings of the previous sections and show the advantage of employing exponential
integrators in our context.
In all tests, we assume that the number of agents is $100$ and that they start uniformly distributed in the interval $[-1,1]$. The starting point of the time horizon is $t_0=0$, while the final simulation time is set to $T=1$.
To model the interactions, if not differently specified, we employ a Cucker--Smale-like kernel of the form
\begin{equation}\label{eq:symmker}
  p(x_k(t),x_\ell(t))=\frac{1}{\epsilon(\alpha^2+\lvert x_k(t)-x_\ell(t) \rvert^2)},
\end{equation}
with $\alpha=0.1$. Notice that the stiffness in the interaction kernel is essentially controlled by the magnitude of the parameter $\epsilon$, which changes depending on the specific example under consideration. Concerning
the stochastic noise component in \eqref{eq:dynamics}, if present, we set
$\sigma=0.01$. Also, the number of Brownian motion paths considered for the numerical
simulations is $20$.

We compare the stochastic exponential Rosenbrock--Euler
method~\eqref{eq:serb} with the well-known
Euler--Maruyama method~\cite{H01}. For the system under consideration, i.e.,
problem~\eqref{eq:dyn_matrix}, we recall that this method marches as
\begin{equation}\label{eq:eulermaru}
  \bb x^{n+1} = \bb x^n + \tau \bb F(t_n,\bb x^n) + \sigma d\bb W^n.
\end{equation}

All the numerical experiments have been performed with
MathWorks MATLAB\textsuperscript{\textregistered} R2022a on an
Intel\textsuperscript{\textregistered} Core\textsuperscript{\texttrademark}
i7-10750H CPU (six physical cores) equipped with 16GB of RAM. To compute the needed
$\varphi$ functions in the exponential integrators, we employ a Pad\'e approximation
with modified scaling and squaring~\cite{SW09}.

\subsection{Examples with no control in the dynamics}\label{sec:numexpnc}
The first test that we perform corroborates the assertion that standard explicit methods (the Euler--Maruyama scheme~\eqref{eq:eulermaru}, in
particular) do require a time step size restriction when the interaction kernel is stiff, while the stochastic exponential Rosenbrock--Euler method~\eqref{eq:serb} does not. To this aim, we set the value $\epsilon=5\cdot 10^{-2}$ and perform the time integration with $m=25$ time steps. We assume that there is no control in the dynamics, i.e., $u_k(t)=0$ for each $k$ and $t$. In this way, we expect that the agents will converge (in mean) to the average value of the initial positions, i.e., zero. 
The results of the experiment are summarized in Figure~\ref{fig:test1_stoch_stiff_a}.
As we can clearly observe, the trajectories computed using the Euler--Maruyama method (top plot) are highly oscillatory, failing to reach in a stable way the zero steady state. Further reducing the number of time steps or the value of $\epsilon$ results in increasingly worse plots. On the other hand, the outcome with the exponential
integrator (bottom plot) is in line with the expected behavior of the system
(even with such a low number of time steps), i.e., the agents eventually reach the zero position.
In fact, to obtain comparable results employing the Euler--Maruyama method, we
need to employ roughly $60$ times more number of time step (see Figure~\ref{fig:test1_stoch_stiff_b})
\begin{figure}[!htb]
  \centering
  \includegraphics[scale=0.5]{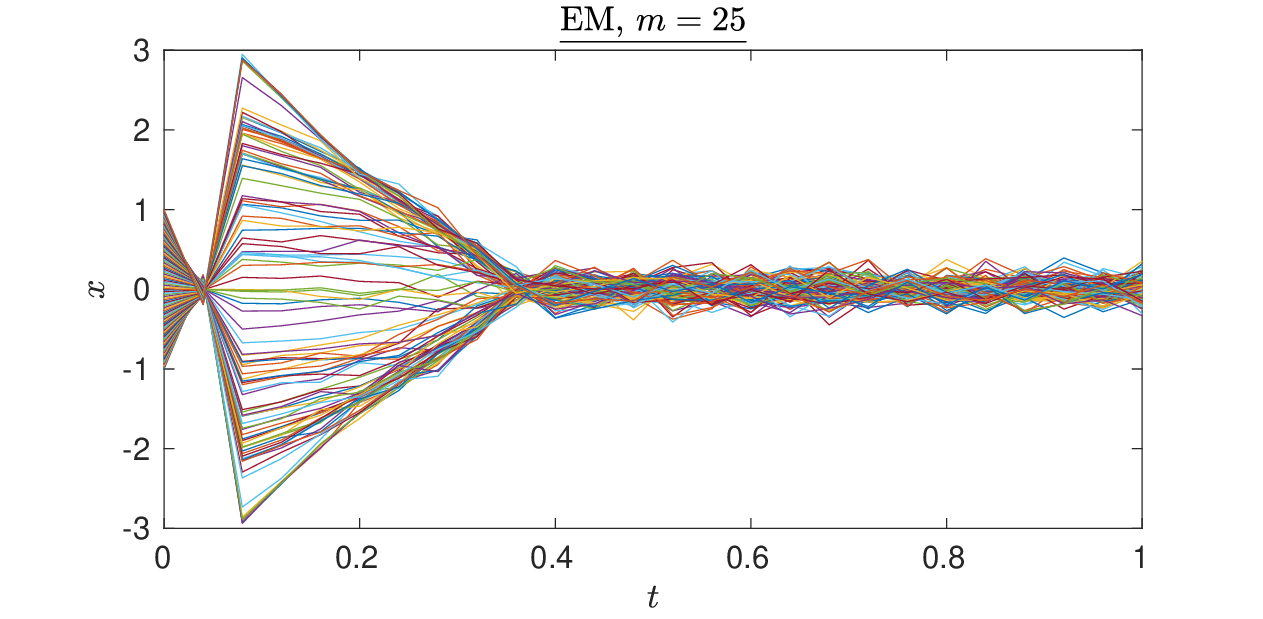}\\[1ex]
  \includegraphics[scale=0.5]{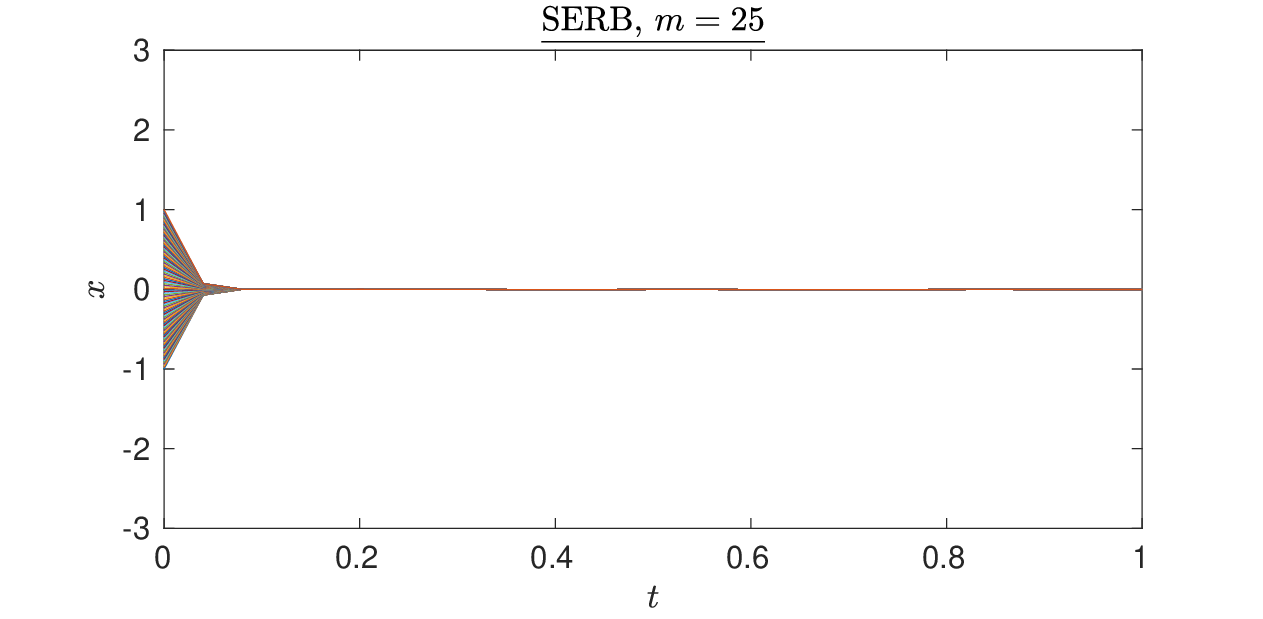}
  \caption{Mean value of the computed trajectories when there is no control in the dynamics and $\epsilon=5\cdot 10^{-2}$.
  The number of time steps is $m=25$.
  Top plot: Euler--Maruyama method (EM). Bottom plot: stochastic exponential Rosenbrock--Euler method (SERB).}
  \label{fig:test1_stoch_stiff_a}
\end{figure}
\begin{figure}[!htb]
  \centering
  \includegraphics[scale=0.5]{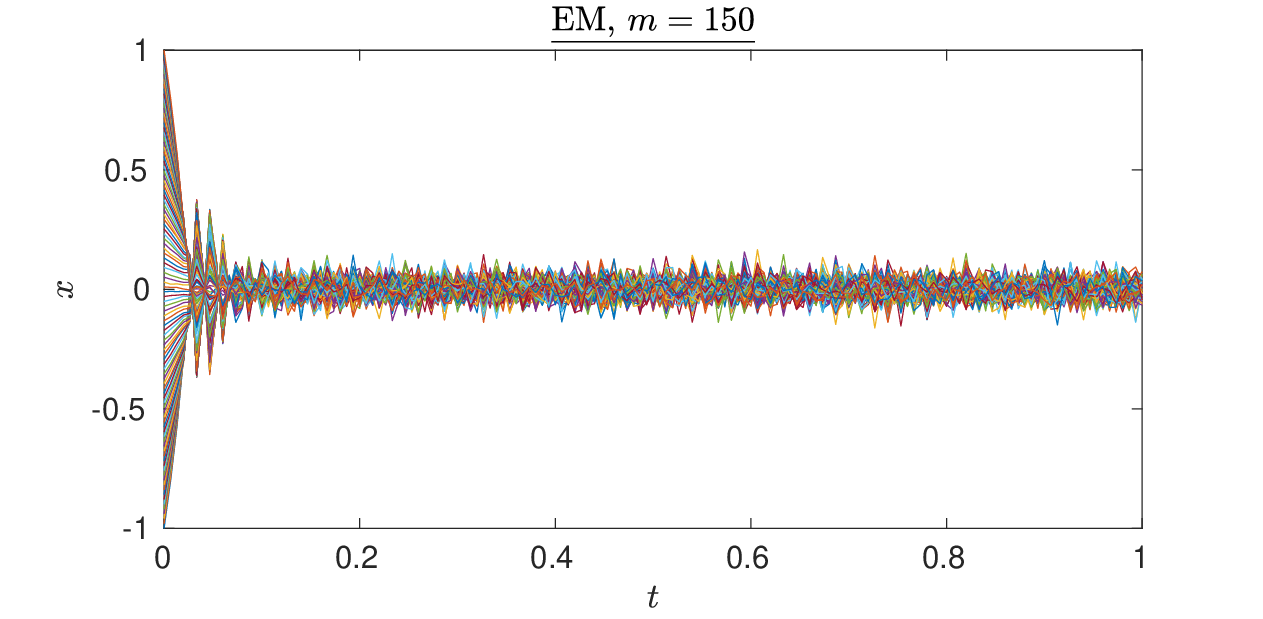}\\[1ex]
  \includegraphics[scale=0.5]{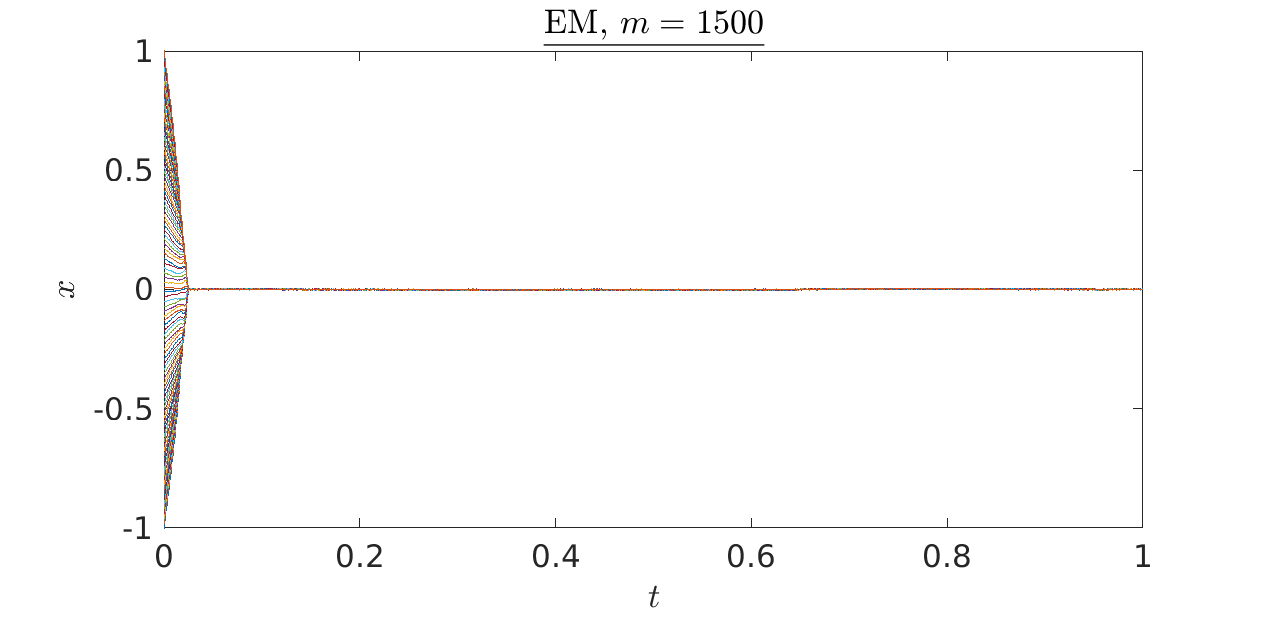}
  \caption{Mean value of the computed trajectories when there is no control in the dynamics and $\epsilon=5\cdot 10^{-2}$.
  Top: Euler--Maruyama method (EM) with $m=150$ time steps. Bottom: Euler--Maruyama method (EM) with $m=1500$ time steps.}
  \label{fig:test1_stoch_stiff_b}
\end{figure}

Remark that the requirement of a large number of time steps for the Euler--Maruyama method is due to the stiffness of the interaction kernel. Indeed, if we perform the experiment setting $\epsilon=1$, we obtain comparable
results (with respect to the stochastic Rosenbrock--Euler method) using
$m=50$ time steps for both schemes (see Figure~\ref{fig:test1_stoch_nonstiff})
\begin{figure}[!htb]
  \centering
  \includegraphics[scale=0.5]{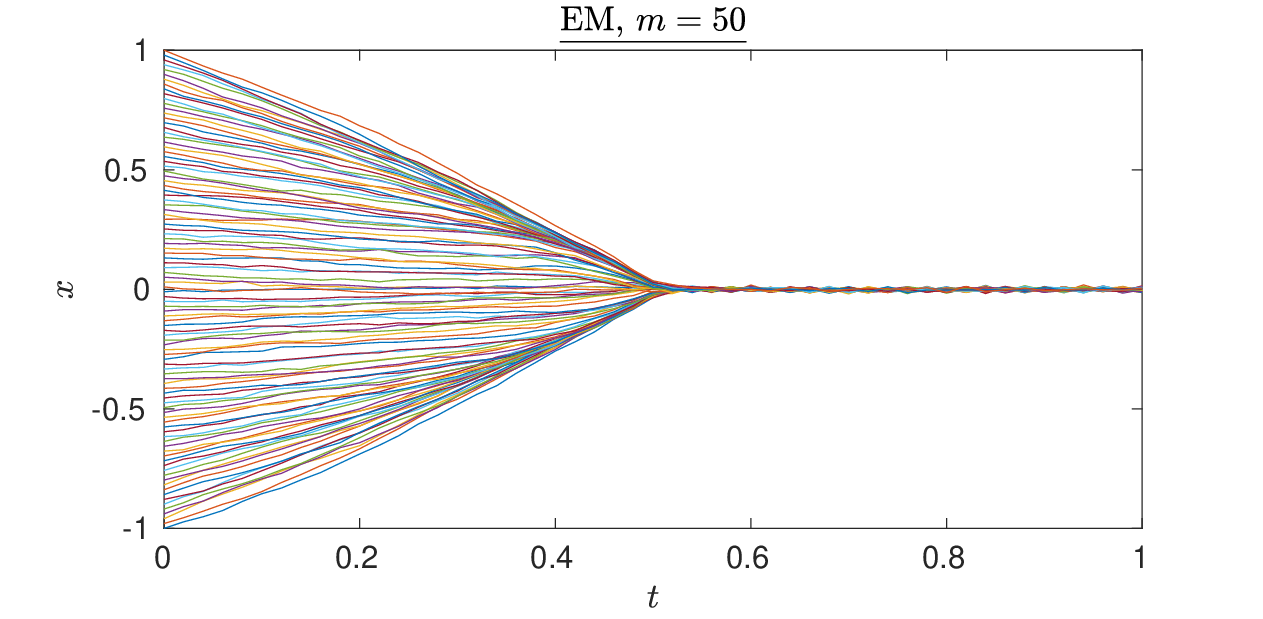}\\[1ex]
  \includegraphics[scale=0.5]{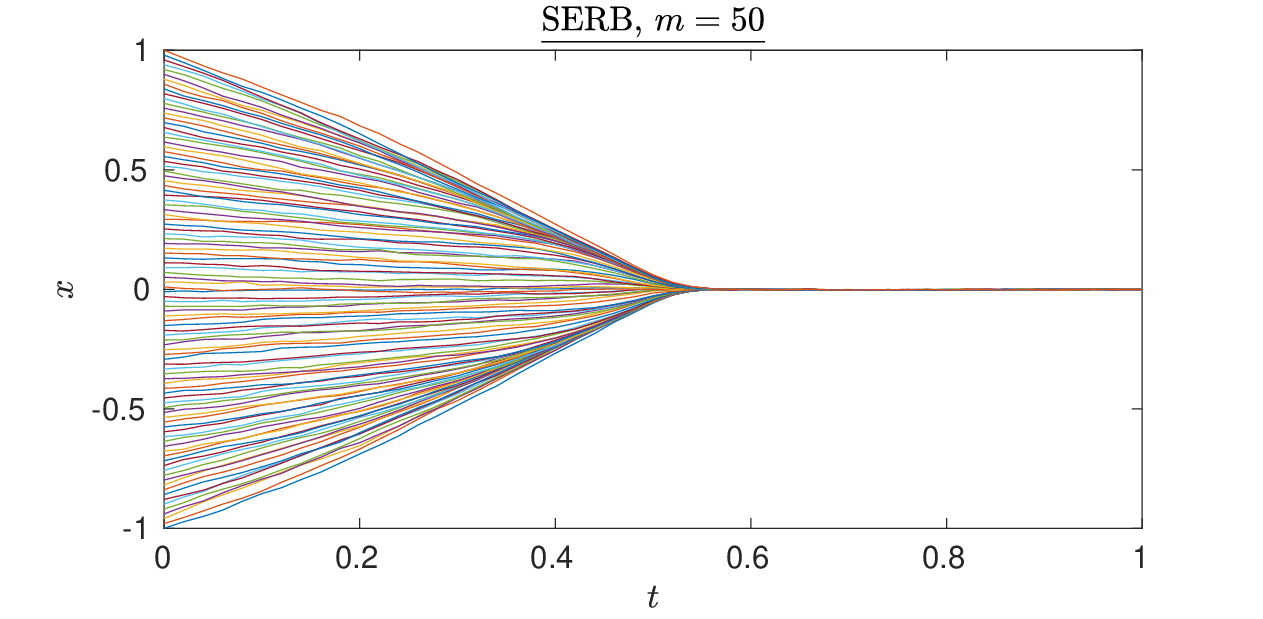}
  \caption{Mean value of the computed trajectories when there is no control in the dynamics and $\epsilon=1$. The number of time steps is $m=50$.
  Top: Euler--Maruyama method (EM). Bottom: stochastic exponential Rosenbrock--Euler method (SERB).}
  \label{fig:test1_stoch_nonstiff}
\end{figure}

We finally notice that we can draw similar conclusions in a deterministic scenario, i.e., when $\sigma=0$. In fact, the randomness in the dynamics does not mitigate the need of a strong time step size restriction for the standard explicit integrators. We demonstrate this by performing a numerical experiment
setting $\epsilon=5\cdot 10^{-2}$. For the time marching, we consider the classical first-order explicit Euler method and the second-order Runge--Kutta scheme
\begin{equation*}
\begin{aligned}
    \bb{X}^{n2} &= \bb{x}^n + \tau\bb{F}(t_n,\bb{x}^n),\\
    \bb{x}^{n+1} &= \bb{x}^n + \frac{\tau}{2}\left(\bb{F}(t_n,\bb{x}^n)+\bb{F}(t_{n+1},\bb{X}^{n2})\right),
\end{aligned}
\end{equation*}
also known as Heun's method or explicit trapezoidal rule. As exponential
integrator, we employ the standard explicit Rosenbrock--Euler method~\eqref{eq:expRB}, which is second-order accurate.
We report in Figures~\ref{fig:test1_deter_stiff_a} and~\ref{fig:test1_deter_stiff_b} the results of the experiments. As expected,
both explicit Euler and Heun's method are not able to correctly capture the dynamics of the system, unless a consistently high number of time steps is employed. This is in contrast with what happens for the exponential Rosenbrock--Euler method, which already with $m=25$ time steps reaches the expected state in a stable way.
\begin{figure}[!htb]
  \centering
  \includegraphics[scale=0.5]{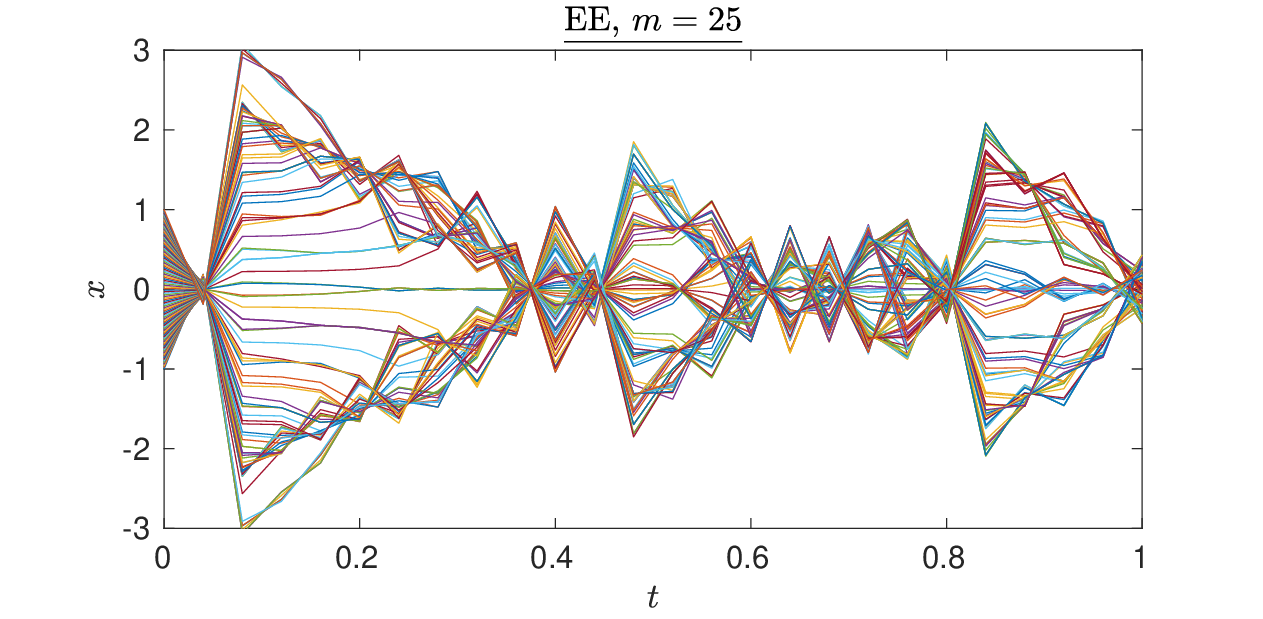}\\[1ex]
  \includegraphics[scale=0.5]{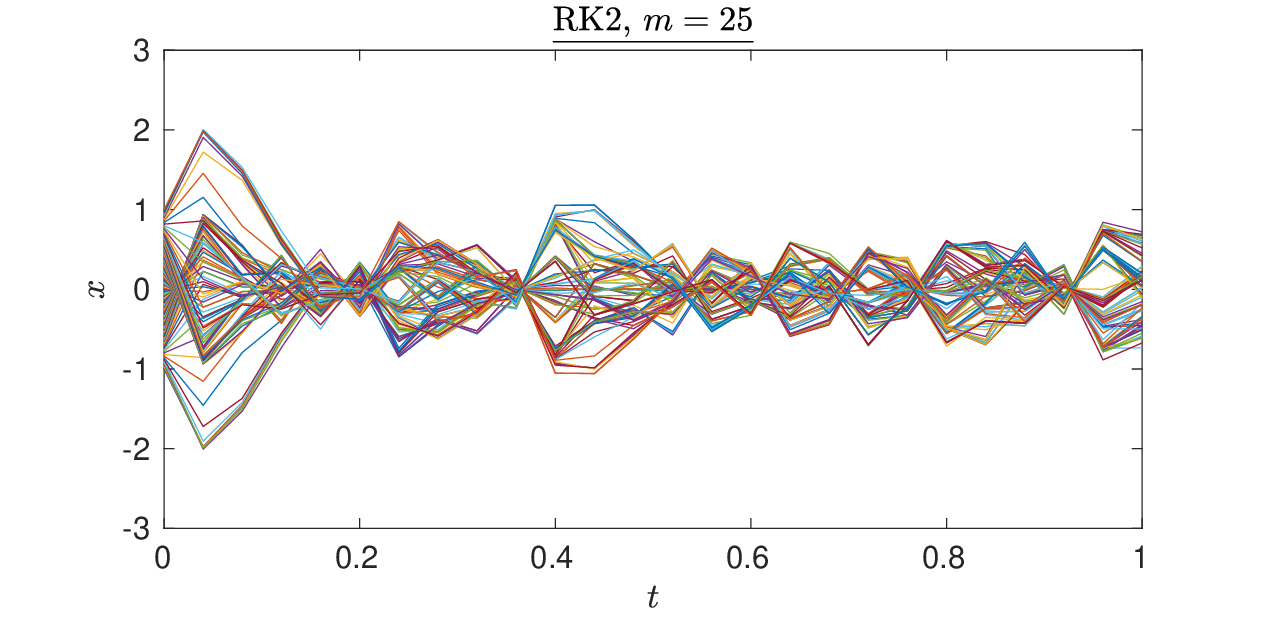}\\[1ex]
  \includegraphics[scale=0.5]{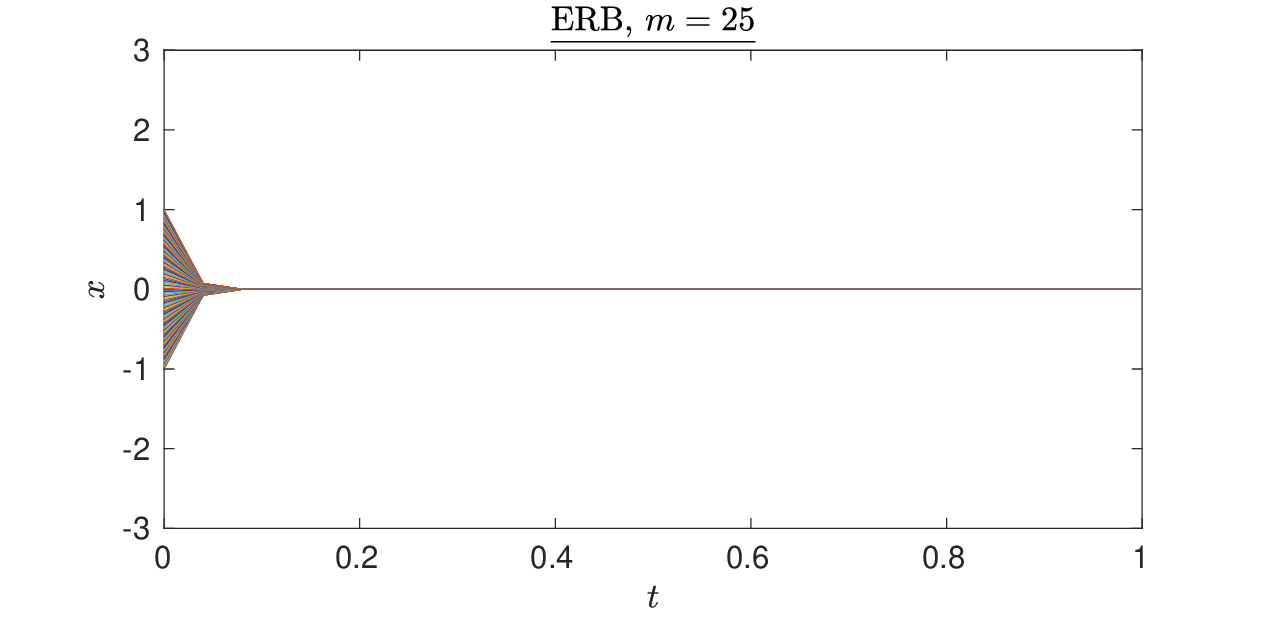}
  \caption{Computed trajectories when there is no control in the dynamics and $\epsilon=5\cdot 10^{-2}$.
  The number of time steps is $m=25$.
  Top plot: explicit Euler method (EE). Center plot: Heun's method (RK2). Bottom plot: exponential Rosenbrock--Euler method (ERB).}
  \label{fig:test1_deter_stiff_a}
\end{figure}
\begin{figure}[!htb]
  \centering
  \includegraphics[scale=0.5]{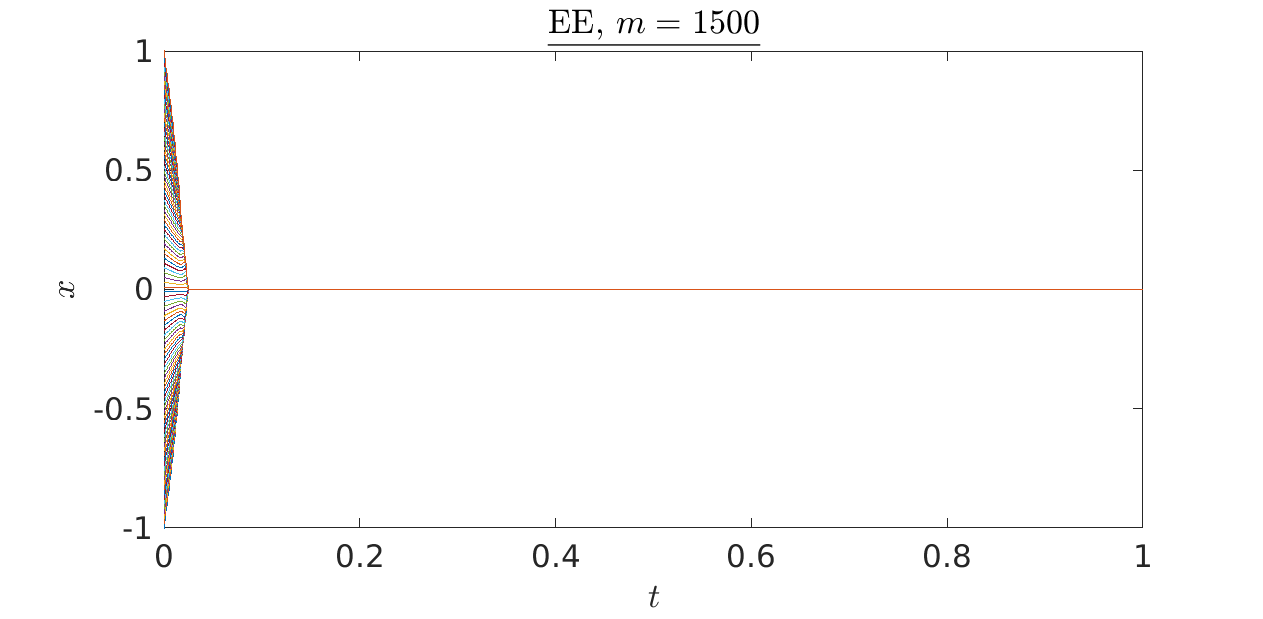}\\[1ex]
  \includegraphics[scale=0.5]{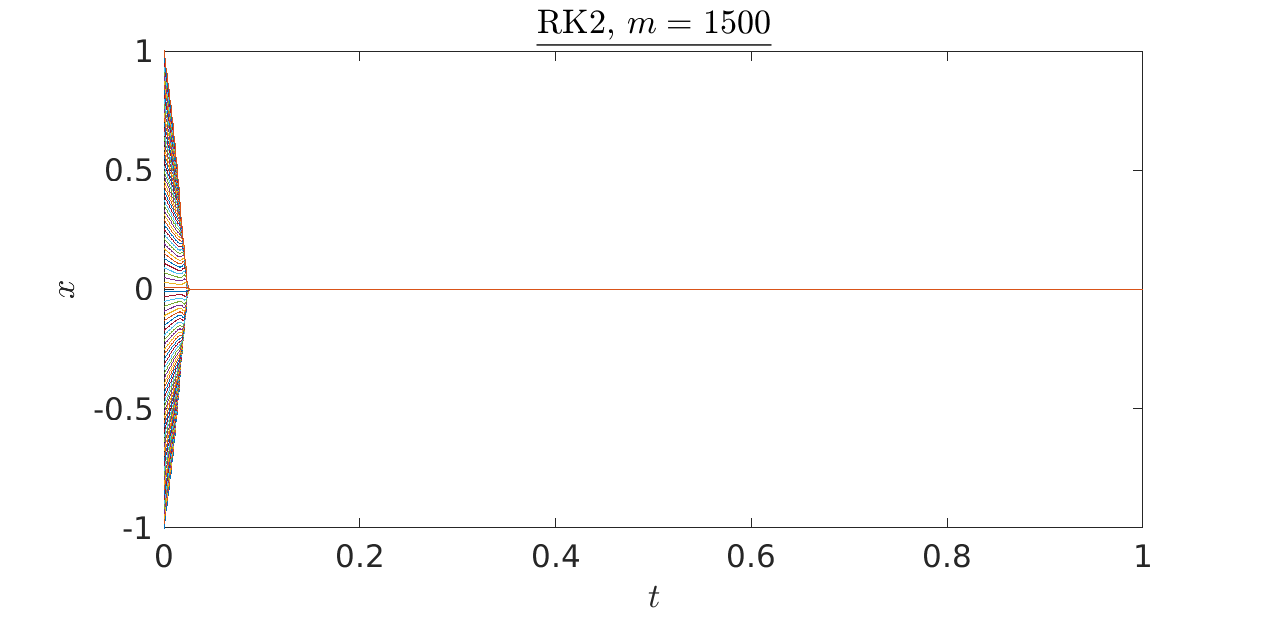}
  \caption{Computed trajectories when there is no control in the dynamics and $\epsilon=5\cdot 10^{-2}$.
  The number of time steps is $m=1500$.
  Top plot: explicit Euler method (EE). Bottom plot: Heun's method (RK2).}
  \label{fig:test1_deter_stiff_b}
\end{figure}

\subsection{Tests with turnpike control}\label{sec:numexptc}
We now perform some numerical examples that show the validity and the effectiveness of the
turnpike property in our setting.
To this aim, we start by considering the interaction kernel~\eqref{eq:symmker}, setting $\epsilon=5\cdot 10^{-2}$. We choose as decay parameter for the Lyapunov functional $\beta=12$ and impose the threshold value $\delta=2\cdot 10^{-4}$ (see Section~\ref{sec:estimate}). Then, according to
formula~\eqref{eq:tptime}, once we set the target point we are able to find an estimate for the turnpike time $\overline{t}$. Remark that, in terms of actual simulation, we split the time interval $[0,T]$ into $I_1=[0,\overline{t}]$ and $I_2=[\overline{t},T]$. In $I_1$ we perform the \textit{cheap} time integration of
equation~\eqref{eq:dyn_mean} by means of the
exponential Rosenbrock--Euler method (see formula~\eqref{eq:expint_mean}) or the explicit Euler method, depending on the time marching framework we are considering.
We stress that when cheap integrating with
the exponential integrator, we could perform a \textit{single} step to reach the time $\overline{t}$ (in contrast to, e.g., the explicit Euler method). This is indeed the case since,
at this stage, the equation is linear, and therefore the exponential integrator
solves it \textit{exactly} (as already discussed in the proof of Lemma~\ref{lem:cheap}). Nevertheless, just for visualization purposes, in the experiment we computed the position of the agents at each time discretization point. 
Finally, in the time interval $I_2$ the static control can be employed, which in our
case is zero. Therefore, in $I_2$ we are essentially evolving the dynamics as if it
was uncontrolled (as done in the experiments in the previous subsection). The number of
time steps is always set to $m=50$.

For the first experiment we choose $\overline{x}=0.7$ as target state. The turnpike time, according to formula~\eqref{eq:tptime}, is then selected as $\overline{t}=0.54$.
The results are reported in Figure~\ref{fig:test2_cheap_stiff} (top and center plot).
As expected, since there is no source of stiffness in the cheap integration part,
the standard explicit method and the exponential integrator qualitatively behave similarly. However, when the turnpike time is surpassed, the system is left uncontrolled
and the instability coming from the stiff kernel arises in the Euler--Maruyama method (as already observed in the previous subsection). This, in turn, does not happen when
working in the exponential integrators framework, which allow to reach the target point in a stable way.
We then conclude that the latter is clearly the preferred approach in this context.

We now test the validity of the exponential integration procedure also when the
target point is outside the initial interval $[-1,1]$. In particular, we set $\overline{x}=-1.7$ and select accordingly the turnpike time as $\overline{t}=0.6$.
The result, presented in Figure~\ref{fig:test2_cheap_stiff} (bottom plot) is in perfect
agreement with what expected.
\begin{figure}[!htb]
  \centering
  \includegraphics[scale=0.5]{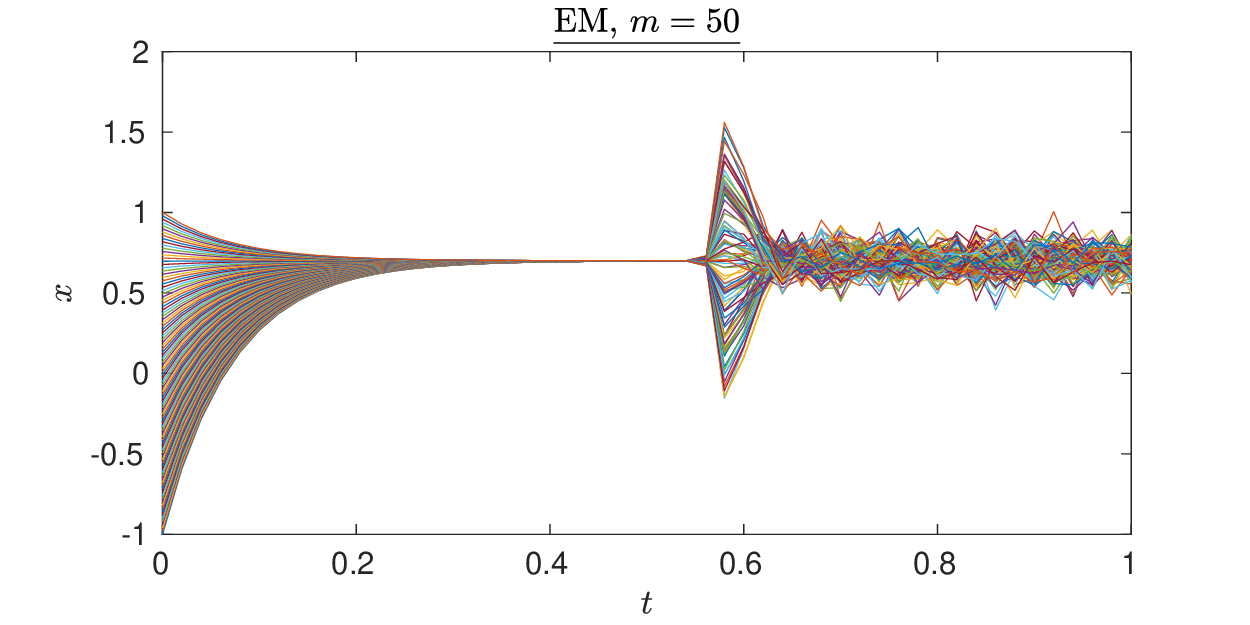}\\[1ex]
  \includegraphics[scale=0.5]{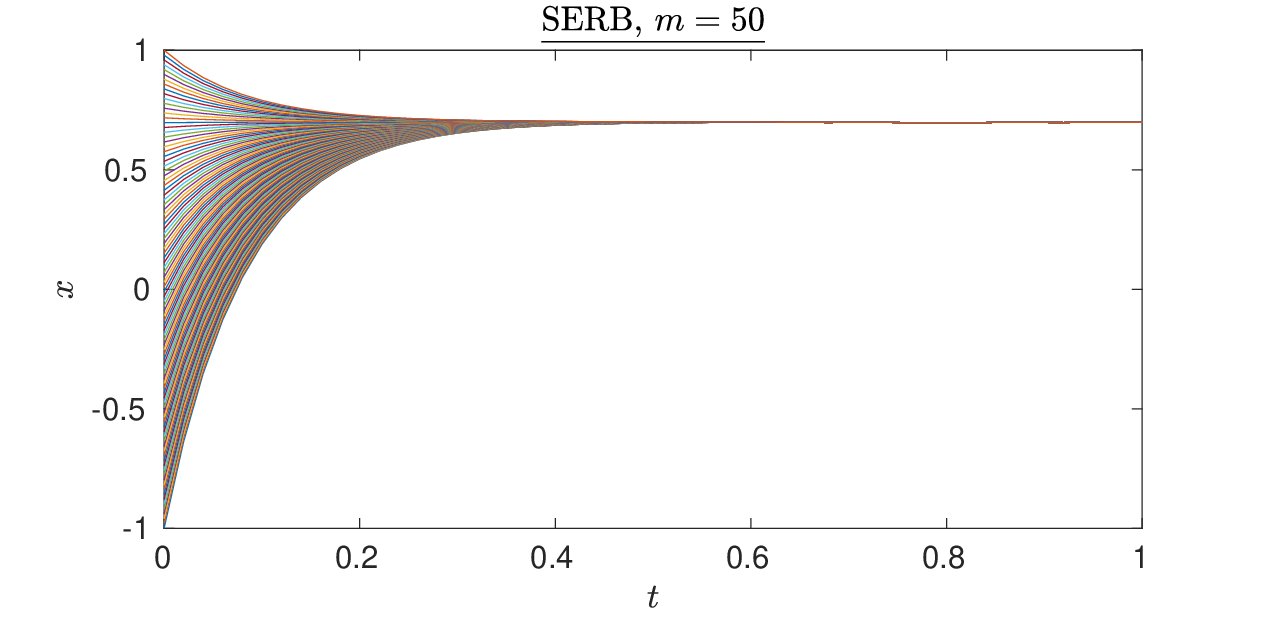}\\[1ex]
  \includegraphics[scale=0.5]{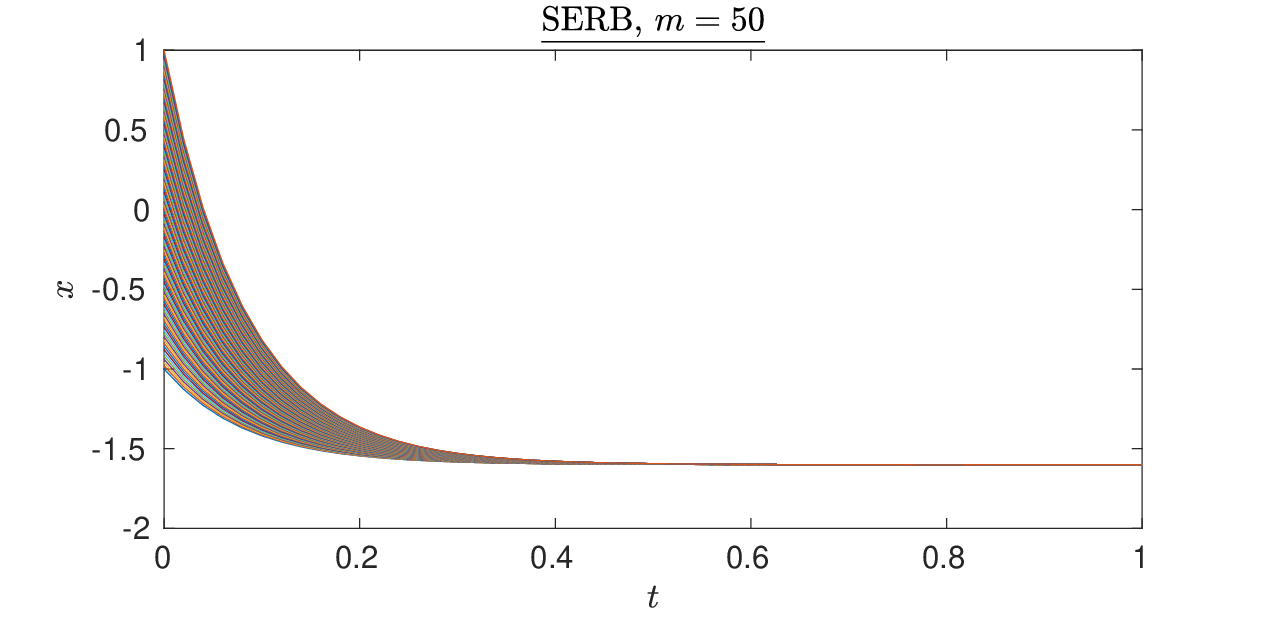}
  \caption{Tests with cheap control, $\epsilon=5\cdot 10^{-2}$.
  The plot is the expected value of the results computed with $n_{bm}=20$
  Brownian motion paths. The number of time steps is $m=50$.
  Top: Euler--Maruyama method (EM), target $\bar{x}=0.7$. Center: stochastic exponential Rosenbrock--Euler method (SERB), target $\bar{x}=0.7$. Bottom: stochastic exponential Rosenbrock--Euler method (SERB), target $\bar{x}=-1.6$.}
  \label{fig:test2_cheap_stiff}
\end{figure}

Finally, we employ the proposed approach in the context of a non-symmetric kernel. To this aim, we de-symmetrize~\eqref{eq:symmker} by considering instead
 \begin{equation}\label{eq:nonsymmker}
    p(x_k(t),x_\ell(t))=\frac{1}{\epsilon_\ell(\alpha^2+\lvert x_k(t)-x_\ell(t) \rvert^2)}
  \end{equation}
  with
  \begin{equation*}
  \epsilon_\ell = \epsilon_\mathrm{min} + \frac{(\epsilon_\mathrm{max}-\epsilon_\mathrm{min})(\ell-1)}{N-1}.
  \end{equation*}
  In fact, the agents are linearly weighted with magnitudes controlled by $\epsilon_\mathrm{min}$ and $\epsilon_\mathrm{max}$. The parameters are set to $\alpha=0.1$, $\epsilon_\mathrm{min}=10^{-2}$, and $\epsilon_\mathrm{max}=10^{-1}$. The result of the experiment with target point $\overline{x}=2.3$ and $\overline{t}=0.62$ (obtained once again setting $\beta=12$ and $\delta=2\cdot10^{-4}$) is summarized in Figure~\ref{fig:test2_cheap_stiff_nonsymm}.
  The outcome is in line with what expected, i.e., the employment of a non-symmetric kernel does not generate issues for the turnpike control and, in particular, does not nullify the convergence of the system to the desired state.
  \begin{figure}[!htb]
  \centering
  \includegraphics[scale=0.5]{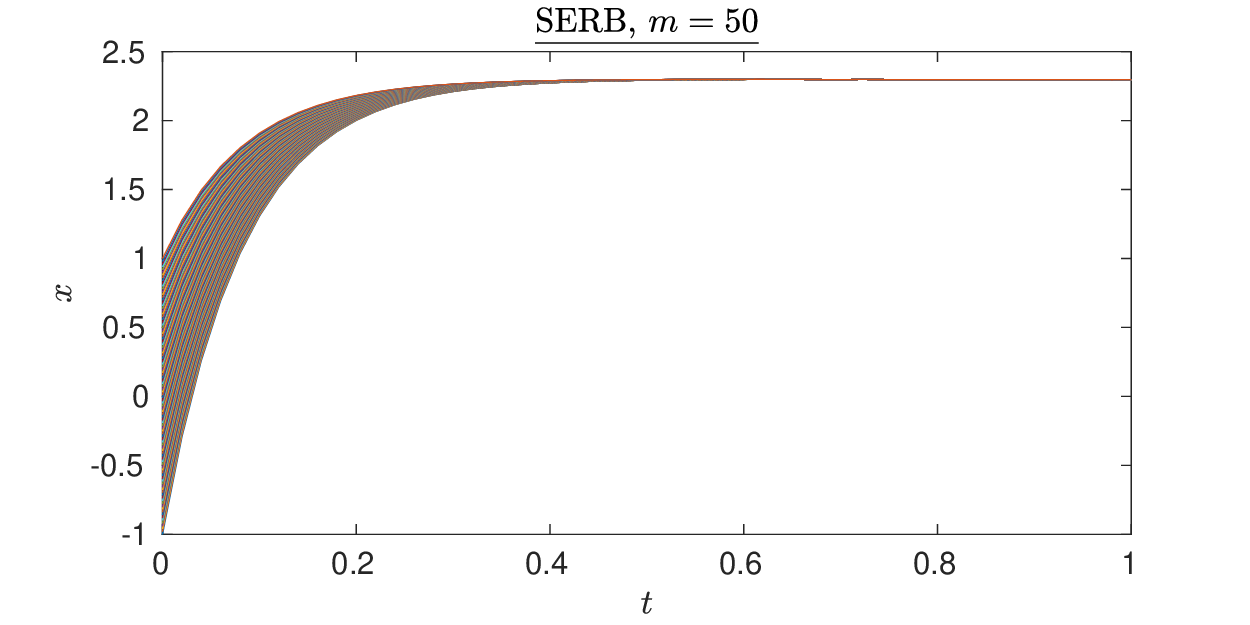}
  \caption{Test with cheap control and non-symmetric interaction kernel~\eqref{eq:nonsymmker}.
  The plot is the expected value of the results computed with $n_{bm}=20$
  Brownian motion paths. The number of time steps is $m=50$.
  The time marching is performed with the stochastic exponential Rosenbrock--Euler method (SERB) with target $\bar{x}=2.3$.}
  \label{fig:test2_cheap_stiff_nonsymm}
\end{figure}

\section{Conclusions}\label{sec:conc}
In this work, we established the turnpike property for discrete-time stochastic optimal control problems governing interacting agents. By leveraging strict dissipativity and cheap control conditions, we showed that the turnpike phenomenon persists despite the presence of noise. To address numerical challenges arising from stiffness, we employed exponential integrators, demonstrating their stability advantages over standard explicit schemes. Our numerical experiments confirmed that the proposed approach effectively maintains the desired asymptotic behavior in several scenarios. For future work, an interesting direction would be to extend the analysis beyond the mean value and investigate the variance of the stochastic dynamics. This would provide a deeper understanding of the fluctuations around the mean and their impact on the control strategy. Additionally, exploring more complex interaction kernels, including unbounded kernels, could lead to further insights into the turnpike behavior of large-scale agent systems. 

\section*{Acknowledgments}
The authors are deeply grateful to Prof.~Michael Herty for its continuous interest in this work and for the numerous constructive discussions.
The authors are members of the Gruppo Nazionale Calcolo Scientifico-Istituto Nazionale
di Alta Matematica (GNCS-INdAM). F. Cassini holds a post-doc fellowship funded by INdAM.
C. Segala thanks the Swiss National Science Foundation (SNSF) for the financial support through the grant number 215528, Large-scale kernel methods in financial economics.

\bibliographystyle{acm}
\bibliography{referencesnmr.bib}
	
\end{document}